\documentclass[graybox]{svmult}

\usepackage{graphicx}
\usepackage{multicol}
\usepackage[bottom]{footmisc}

\usepackage{cite,url}
\usepackage{times}

\usepackage{amssymb}
\usepackage{amsmath}
\usepackage{amsfonts}
\usepackage{arydshln}

\newcommand{\an}[1]{\textcolor{black}{#1}}
\def\bx{{\mathbf{x}}}

\newtheorem{assumption}{Assumption}

\providecommand{\rset}[1]{\mathbb{R}^}
\providecommand{\abs}[1]{\lvert#1\rvert}
\providecommand{\norm}[1]{\lVert#1\rVert}

\begin{document}

\title*{Complexity certifications of
first order inexact Lagrangian methods for general convex programming}

\titlerunning{Complexity certifications of  first order inexact dual methods}

\author{Ion Necoara, Andrei Patrascu  and \an{Angelia Nedi\'c} }

\institute{I. Necoara \and A. Patrascu \at Automatic Control and
Systems Engineering Department, University  Politehnica Bucharest,
Romania, \email{ion.necoara,andrei.patrascu@acse.pub.ro}.  Second
author's work has been funded by the Sectoral Operational Programme
Human Resources Development 2007-2013 of the Ministry of European
Funds through the Financial Agreement POSDRU/159/1.5/S/134398. \and
A. Nedi\'c \at Industrial and Enterprise Systems Engineering
Department, University of Illinois at Urbana-Champaign, USA,
\email{angelia@illinois.edu}. Third author's work has been funded by
the Office of Naval Research under grant no. N00014-12-1-0998.  }

\maketitle

\abstract{ In this chapter we derive computational complexity
certifications of first order inexact dual methods for solving
general smooth \an{constrained convex problems which can arise in
real-time applications, such as} model predictive control. When it
is difficult to project on the primal \an{constraint set described
by a collection of general convex functions,} we use the Lagrangian
relaxation to handle the complicated constraints and then, we apply
dual (fast) gradient algorithms based on inexact dual gradient
information for solving the corresponding dual problem. The
iteration complexity analysis is based on two types of approximate
primal solutions:  the primal last iterate and an average of primal
iterates. We provide sublinear computational complexity estimates on
the primal suboptimality and \an{constraint (feasibility)} violation
of the generated approximate primal solutions. In the final part of
the chapter, we present an open-source quadratic optimization
solver, \an{referred to as DuQuad, for convex quadratic programs and
for evaluation of its behavior. The solver contains the C-language
implementations of the analyzed algorithms.} }

\section{Introduction}\label{in:sec1}
\noindent Nowadays, many engineering applications can be  posed as
general smooth \an{constrained} convex problems.  Several important
applications that can be modeled in this framework \an{have
attracted great attention lately, such as model predictive control
for dynamical linear systems and its dual (often  referred to as
moving horizon estimation)}
\cite{NecFer:14c,NedNec:14,PatBem:12j,RicJon:12,RawMay:09}, DC
optimal power flow problem for \an{power systems} \cite{ZimMur:11},
\an{and} network utility maximization problems \cite{WeiOzd:10}.
\an{Notably}, the recent advances in hardware and numerical
optimization made it possible to solve linear model  predictive
control problems of nontrivial sizes within microseconds even on
hardware platforms with limited computational power and memory.

\noindent In this chapter, we are particularly interested in
real-time linear model predictive control (MPC) problems. For  MPC,
the corresponding optimal control problem can be recast as a smooth
\an{constrained} convex optimization problem.  There are numerous
ways in which this problem can be solved. For example, an interior
point method has been proposed in \cite{RaoWri:98} and an active set
method was described in \cite{FerKir:14}. Also, explicit MPC \an{has
been proposed in \cite{BemMor:02},} where the optimization problem
is solved off-line for all possible states. In \an{real-time (or
on-line)} applications, these methods can sometimes fail due to
their overly complex iterations in the case of interior point  and
active set  methods, or due to the large dimensions of the problem
in the case of explicit \an{MPCs}. Additionally, when embedded
systems are employed, computational complexities need to be kept to
a minimum. As a result, second order algorithms (e.g. interior
point), \an{which} most often require matrix inversions, are usually
left out. \an{In such applications, first}  order algorithms are
more suitable
\cite{NecFer:14c,NecNed:14b,NedNec:14,PatBem:12j,RicJon:12}
\an{especially for instances} when computation power and memory is
limited. For many optimization problems arising in engineering
applications, such as real-time \an{MPCs}, the constraints are
overly complex, making projections on these sets \an{computationally
prohibitive}. This is most often the main impediment of applying
first order methods on the primal optimization problem. To
circumvent this, the dual approach is considered \an{by forming the
dual problem, whereby the complex constraints are moved into the
objective function, thus rendering} much simpler constraints for the
dual variables, often being only the non-negative orthant.
\an{Therefore}, we consider dual first order methods for solving the
dual problem. The computational complexity certification of
gradient-based methods for solving the (augmented) Lagrangian dual
of a primal convex problem is studied e.g. in
\cite{BecNed:14,DevGli:11,KosNed,NecPat:14,NecNed:14b,
NecNed:14a,NedNec:14,NedOzd:08,PatBem:12j}.   However,  these papers
either threat quadratic problems \cite{PatBem:12j} or linearly
constrained smooth convex problems with simple objective function
\cite{BecNed:14,NecPat:14}, or  the approximate primal solution is
generated through  averaging
\cite{NecNed:14b,NecNed:14a,NedNec:14,NedOzd:08}. On the other hand,
in practice usually the last primal  iterate is employed. There are
few attempts to derive the computational complexity of dual gradient
based methods using as an approximate primal solution the last
iterate of the algorithm for particular cases of convex problems
\cite{NecNed:14a,NecPat:14,BecNed:14}. Moreover,  from our practical
experience we have observed that usually these methods converge
faster in the primal last iterate than in a primal  average
sequence. These issues motivate our work here.

\vspace{2pt}

\noindent \textit{Contribution}.
In this chapter,  we analyze the
computational complexity of dual first order methods for solving
general smooth \an{constrained} convex problems. Contrary to most of the results from
the literature \cite{BecNed:14,NecPat:14,NecNed:14a,
NedOzd:08,PatBem:12j}, our approach allows us to use inexact dual
gradient information.    Another important feature of our approach
is that we \an{also} provide complexity results for the primal latest
iterate, while in much of the previous literature convergence rates
in an average of primal iterates are given. This feature is of
practical importance since usually the primal last iterate is
employed in applications. More precisely, the main contributions in
this chapter are:

\vspace{2pt}

\noindent $(i)$ We derive the computational complexity  of the dual
gradient method in terms of primal suboptimality and feasibility
violation using inexact dual gradients and two types of approximate
primal solutions: $\mathcal{O}\left(\frac{1}{\epsilon^2} \log
\frac{1}{\epsilon} \right)$ in the primal last iterate and
$\mathcal{O} \left(\frac{1}{\epsilon} \log \frac{1}{\epsilon}
\right)$ in an average of  primal iterates, where $\epsilon$ is some
desired accuracy.

\noindent $(ii)$ We also derive the computational complexity of the
dual  fast gradient method in terms of primal suboptimality and
feasibility violation using inexact dual gradients and two types of
approximate primal solutions: $\mathcal{O}\left(\frac{1}{\epsilon}
\log \frac{1}{\epsilon}\right)$ in the primal  last iterate and
$\mathcal{O} \left(\frac{1}{\sqrt{\epsilon}} \log
\frac{1}{\epsilon}\right)$ in a primal average sequence.

\noindent $(iii)$ Finally, we present an open-source  optimization
solver, \an{termed} DuQuad, consisting of the C-language implementations
of the above inexact dual first order  algorithms for solving
convex quadratic problems, and \an{we study} its numerical behavior.

\vspace{2pt}

\noindent \textit{\an{Content}}.
The chapter is organized as follows. In
Section 2 we formulate our problem of interest and its dual, and \an{we}
analyze its smoothness property. In Section 3 we introduce a general
inexact dual first order method, covering the inexact dual gradient
and fast gradient algorithms, and  \an{we } derive computational complexity
certificates for these schemes. Finally, in Section 4 we describe
briefly the DuQuad toolbox \an{that implements the above
inexact algorithms for solving convex quadratic programs in C-language}, while in
Section 5 we provide detailed numerical experiments.

\vspace{2pt}

\noindent \textit{\an{Notation}}.
We consider the space $\rset^n$
composed  \an{of} column vectors.  For $\mathbf{x}, \mathbf{y} \in \rset^n$, \an{we} denote the
scalar product by $\langle \mathbf{x},\mathbf{y} \rangle = \mathbf{x}^T \mathbf{y}$ and the  Euclidean
norm by $\|\mathbf{x}\|=\sqrt{\mathbf{x}^T \mathbf{x}}$. We denote the nonnegative orthant by
$\rset^n_{+}$ and \an{we use $[\mathbf{u}]_{+}$ for}  the projection of $\mathbf{u}$ onto
$\rset^n_+$. The minimal eigenvalue of a symmetric matrix $\mathbf{Q} \in
\rset^{n \times n}$ is denoted by $\lambda_{\min}(\mathbf{Q})$ and $\|\mathbf{Q}\|_F$
denotes its Frobenius norm.

\def\b0{{\mathbf{0}}} 


\section{Problem formulation}
\label{in:sec3}
\noindent In this \an{section}, we consider
the following general \an{constrained} convex optimization problem:
\begin{align}
\label{in:eq_prob_princc1}
 f^* = \min_{\mathbf{u} \in U}& \;\;  f(\mathbf{u})  \qquad \text{s.t.:}  \;\; \mathbf{g}(\mathbf{u}) \leq \b0,
\end{align}
where $U \subseteq \rset^n$ is a closed simple convex set (e.g. a box set), $\b0\in\rset^p$ is a vector of zeros,
and \an{the constraint mapping $\mathbf{g}(\cdot)$
is given by $\mathbf{g}(\cdot) =\an{[g_1(\cdot), \ldots,}g_p(\cdot)]^T$.
(The vector inequality $\mathbf{g}(\mathbf{u}) \leq \b0$ is to be understood coordinate-wise.)
The objective function $f(\cdot)$ and the constraint functions $g_1(\cdot), \ldots, g_p(\cdot)$ are convex and differentiable
over their domains.}
Many engineering
applications can be posed as the general convex problem
\eqref{in:eq_prob_princc1}. \an{For example for} linear model predictive control
problem in condensed form
\cite{NecFer:14c,NedNec:14,PatBem:12j,RicJon:12,RawMay:09}: $f$ is
convex (quadratic) function, $U$ is  box set describing the input
constraints and $\mathbf{g}$ is \an{given by convex functions} describing the state
constraints; \an{for } network utility maximization problem \cite{BecNed:14}:
$f$ is $\log$ function, $U =\rset^n_+$ and $\mathbf{g}$ is linear function
describing the link capacities; \an{for } DC optimal power flow problem
\cite{ZimMur:11}: $f$ is convex function, $U$ is  box set and $\mathbf{g}$
describes the DC nodal power balance constraints.

\noindent We are interested in deriving computational complexity estimates of dual first order methods for solving the optimization problem \eqref{in:eq_prob_princc1}. We make the following assumptions on the objective function and the feasible set
of the problem~\eqref{in:eq_prob_princc1}.

\begin{assumption}
\label{in:assump_all}
\an{
Let $U \subseteq \text{dom} \ f \cap
\left\{ \cap_{i=1}^p \text{dom} \ g_i \right\} $, and assume that:}

\vspace*{3pt}
\noindent
\an{
(a)  The Slater condition holds for the
feasible set  of problem \eqref{in:eq_prob_princc1}, i.e.,
there exists $\bar{\mathbf{u}} \in \text{relint}(U)$ such that $ \mathbf{g}(\bar{\mathbf{u}})
<\b0$.}

\vspace*{3pt}
\noindent (b)  The function $f$ is
strongly convex with constant $\sigma_f >0$ and has Lipschitz
continuous \an{gradients} with constant $L_f$, i.e.:
\begin{equation*}
 \frac{\sigma_f}{2}\norm{\mathbf{u}- \mathbf{v}}^2 \le f(\mathbf{u}) - \left( f(\mathbf{v}) + \langle
  \nabla f(\mathbf{v}), \mathbf{u}- \mathbf{v} \rangle \right) \le \frac{L_f}{2} \norm{\mathbf{u} - \mathbf{v}}^2
   \quad \forall \mathbf{u}, \mathbf{v} \in U.
\end{equation*}
\noindent (c)  The function $\mathbf{g}$ has bounded Jacobians on the set
$U$, i.e.,  there exists  $c_g > 0$ such that $\|\nabla \mathbf{g}(\mathbf{u})\|_F \leq
c_g$ \;  for all $\mathbf{u} \in U$.
\end{assumption}

\noindent Moreover, we introduce the following definition:
\begin{definition}
Given $\epsilon>0$, a primal point $\mathbf{u}_{\epsilon} \in U$ is called $\epsilon$-optimal if it satisfies:
 \begin{equation*}
  | f(\mathbf{u}_{\epsilon}) - f^*| \le \epsilon \quad \text{and} \quad \left\| \left[\mathbf{g}(\mathbf{u}_{\epsilon}) \right]_+ \right\| \le \epsilon.
 \end{equation*}
\end{definition}

\noindent
Since $U$ is assumed to be a simple set, i.e. the
projection on this set is easy (e.g. a box set),  we denote the associated dual
problem of \eqref{in:eq_prob_princc1} as:
\begin{equation}\label{in:eq_dual_prob}
\max_{\mathbf{x} \ge 0} \; d(\mathbf{x}) \quad \left( = \min_{\mathbf{u} \in U} \mathcal{L}(\mathbf{u},\mathbf{x}) \right),
\end{equation}
where the Lagrangian function is given by:
\begin{equation*}
\mathcal{L}(\mathbf{u},\mathbf{x}) = f(\mathbf{u}) + \langle \mathbf{x}, \mathbf{g}(\mathbf{u})\rangle.
\end{equation*}
We denote the dual optimal set with  \an{$X^* =
\arg\max\limits_{\mathbf{x} \ge \b0} d(\mathbf{x})$}.
Note that Assumption~\ref{in:assump_all}$ (a)$ guarantees that strong duality holds for
\eqref{in:eq_prob_princc1}. Moreover, since $f$ is strongly convex
function (see Assumption \ref{in:assump_all}$ (b)$), the inner subproblem
$\min\limits_{\mathbf{u} \in U} \mathcal{L}(\mathbf{u},\mathbf{x})$   has the objective
function $\mathcal{L}(\cdot,\mathbf{x})$ strongly convex for any fixed $\mathbf{x} \in \rset^{p}_+$. It follows that the
optimal solution $\mathbf{u}^*$ of the original problem
\eqref{in:eq_prob_princc1}  and  $\mathbf{u}(\mathbf{x}) = \arg \min_{\mathbf{u} \in U}
\mathcal{L}(\mathbf{u},\mathbf{x})$ are unique and, thus, from Danskin's theorem
\cite{Nes:05} we get that the dual function $d$ is differentiable
on $\rset^n_+$ and its gradient is given by:
\[\nabla d(\mathbf{x}) = \mathbf{g}(\mathbf{u}(\mathbf{x}))\qquad\hbox{for all $\mathbf{x} \in \rset^n_+$}.\]
From Assumption~\ref{in:assump_all}$(c)$ it follows immediately, using the mean value
theorem, that the function $\mathbf{g}$ is Lipschitz continuous with
constant $c_g$, i.e.,
\begin{equation}\label{in:lip_g}
\|\mathbf{g}(\mathbf{u}) - \mathbf{g}(\mathbf{v})\| \leq    c_g \|\mathbf{u}-\mathbf{v}\|  \qquad \forall \mathbf{u}, \mathbf{v} \in U.
\end{equation}
\an{In the forthcoming lemma, Assumption \ref{in:assump_all} $(b)$ and $(c)$ allow us to show}
that the dual function $d$ has Lipschitz
gradient. Our result  is a generalization of a result in
\cite{Nes:05} given there for the case of \an{a linear mapping $\mathbf{g}(\cdot)$}
(see also \cite{NecNed:14b} for a different proof):

\begin{lemma}
\label{in:dual_lips}
Under  Assumption~\ref{in:assump_all},  the dual function \an{$d(\cdot)$} corresponding
to general convex  problem \eqref{in:eq_prob_princc1}
has Lipschitz continuous gradient with constant $L_{\text{d}} =
c_g^2/\sigma_f$, i.e.,
\begin{equation}
\label{in:lipd}
\| \nabla d(\mathbf{x}) - \nabla d(\bar{\mathbf{x}})\|   \leq
c_g^2/\sigma_f \| \mathbf{x} - \bar{\mathbf{x}}\| \qquad \forall \mathbf{x}, \bar{\mathbf{x}}
\in \rset^p_+.
\end{equation}
\end{lemma}

\begin{proof}
Let $\mathbf{x} ,\bar{\mathbf{x}} \in \rset^p_+$. \an{Then, by} using the optimality conditions
for $\mathbf{u}(\mathbf{x})$ and $\mathbf{u}(\bar{\mathbf{x}})$, we get:
\begin{align*}
& \left\langle  \nabla f(\mathbf{u}(\mathbf{x})) + \sum_{i=1}^p \mathbf{x}_i \nabla g_i(\mathbf{u}(\mathbf{x})),  \mathbf{u}(\bar{\mathbf{x}}) - \mathbf{u}(\mathbf{x}) \right\rangle  \geq 0,\\
& \left\langle  \nabla f(\mathbf{u}(\bar{\mathbf{x}})) + \sum_{i=1}^p \bar{\mathbf{x}}_i \nabla g_i(\mathbf{u}(\bar{\mathbf{x}})), \mathbf{u}(\mathbf{x}) - \mathbf{u}(\bar{\mathbf{x}}) \right\rangle  \geq 0.
\end{align*}
Adding \an{these two} inequalities and using  the strong convexity of $f$, we \an{further obtain}
\begin{align*}
&\sigma_f  \| \mathbf{u}(\mathbf{x}) -  \mathbf{u}(\bar{\mathbf{x}})\|^2  \leq  \langle  \nabla f(\mathbf{u}(\mathbf{x})) - \nabla f(\mathbf{u}(\bar{\mathbf{x}})), \mathbf{u}(\mathbf{x}) - \mathbf{u}(\bar{\mathbf{x}}) \rangle \\
& \leq \left \langle  \sum_i \mathbf{x}_i \nabla g_i(\mathbf{u}(\mathbf{x})) - \sum_i \bar{\mathbf{x}}_i \nabla g_i(\mathbf{u}(\bar{\mathbf{x}})), \mathbf{u}(\bar{\mathbf{x}})  - \mathbf{u}(\mathbf{x}) \right \rangle  \\
& =  \left \langle  \an{\sum_{i=1}^p} (\mathbf{x}_i - \bar{\mathbf{x}}_i) \nabla g_i(\mathbf{u}(\mathbf{x})) \!
-\! \an{\sum_{i=1}^p} \bar{\mathbf{x}}_i (\nabla g_i(\mathbf{u}(\bar{\mathbf{x}})) -
\nabla g_i(\mathbf{u}(\mathbf{x}))), \mathbf{u}(\bar{\mathbf{x}})  - \mathbf{u}(\mathbf{x}) \right \rangle \\
& \leq
\left \langle \! \an{\sum_{i=1}^p}
(\mathbf{x}_i \!-\! \bar{\mathbf{x}}_i) \nabla g_i(\mathbf{u}(\mathbf{x})), \mathbf{u}(\bar{\mathbf{x}}) \! -\! \mathbf{u}(\mathbf{x}) \!\!\right \rangle,
\end{align*}
\an{where the last inequality follows from
the convexity of the function $g_i$  and $\bar{\mathbf{x}}_i\ge0$ for all $i$. By the Cauchy-Schwarz inequality, we have}
\begin{align*}
\sigma_f  \| \mathbf{u}(\mathbf{x}) -  \mathbf{u}(\bar{\mathbf{x}})\|^2
&\leq
\an{\sum_{i=1}^p} |\mathbf{x}_i \!-\! \bar{\mathbf{x}}_i| \|\nabla g_i(\mathbf{u}(\mathbf{x}))\| \|\mathbf{u}(\mathbf{x})  \!-\! \mathbf{u}(\bar{\mathbf{x}})\|  \\
& \leq
\| \mathbf{x} - \bar{\mathbf{x}}\| \|\nabla \mathbf{g}(\mathbf{u}(\mathbf{x}))\|_F \|\mathbf{u}(\mathbf{x})  - \mathbf{u}(\bar{\mathbf{x}})\|  \\
&\leq c_g  \| \mathbf{x} - \bar{\mathbf{x}}\|
\|\mathbf{u}(\mathbf{x})  - \mathbf{u}(\bar{\mathbf{x}})\|,
\end{align*}
\an{where the second inequality follows by H\"older's inequality and the last inequality follows by the bounded Jacobian assumption for $\mathbf{g}$ (see Assumption~\ref{in:assump_all}(c$)$).}
Thus, we obtain:
\[ \| \mathbf{u}(\mathbf{x}) -  \mathbf{u}(\bar{\mathbf{x}})\|  \leq  \frac{c_g}{\sigma_f}  \| \mathbf{x} - \bar{\mathbf{x}}\|. \]
Combining  \eqref{in:lip_g} with the \an{preceding} relation,  we obtain that
the gradient of the dual function is Lipschitz continuous with
constant \an{$L_\text{d} = \frac{c_g^2}{\sigma_f}$, i.e.,}
\begin{align*}
\| \nabla d(\mathbf{x}) - \nabla d(\bar{\mathbf{x}}) \|
& = \|\mathbf{g}(\mathbf{u}(\mathbf{x})) - \mathbf{g}(\mathbf{u}(\bar{\mathbf{x}})) \| \le
c_g \|\mathbf{u}(\mathbf{x})- \mathbf{u}(\bar{\mathbf{x}})\| \leq \frac{ c_g^2}{\sigma_f} \| \mathbf{x} - \bar{\mathbf{x}}\|,
\end{align*}
for all $\mathbf{x}, \bar{\mathbf{x}} \in  \rset^p_+$.
\qed
\end{proof}

\noindent
Note that in the case \an{of a linear mapping $\mathbf{g}$},
i.e., $\mathbf{g}(\mathbf{u}) = \mathbf{Gu} + \mathbf{g}$,
\an{we have} $c_g=\|\mathbf{G}\|_F \geq \|\mathbf{G}\|$. In conclusion,
\an{our estimate on the Lipschitz constant of the
gradient of the dual function for general convex constraints}
$L_{\text{d}} = c_g^2/\sigma_f$ can coincide with the one
derived in \cite{Nes:05} for the linear case $L_{\text{d}} =\|\mathbf{G}\|^2/\sigma_f$ \an{if one takes
the linear structure of $\mathbf{g}$ into account in the proof of Lemma~\ref{in:dual_lips}
(specifically, where we used H\"older's inequality)}.
Based \an{on relation \eqref{in:lipd} of Lemma~\ref{in:dual_lips},}
the following descent lemma holds with $L_{\text{d}} =
c_g^2/\sigma_f$ (see for example \cite{Nes:05}):
\begin{equation}\label{in:eq_descent}
d(\mathbf{x}) \ge d(\mathbf{y}) + \langle \nabla d(\mathbf{y}), \mathbf{x}-\mathbf{y}\rangle - \frac{L_{\text{d}}}{2}\norm{\mathbf{x}-\mathbf{y}}^2 \quad \forall \mathbf{x},\mathbf{y} \in \rset^{p}_+.
\end{equation}
Using these preliminary results, \an{in a unified manner},
we analyze further the computational complexity of inexact dual first order methods.


\section{Inexact dual first order methods}
\label{in:sec4} \noindent In this section we introduce and analyze
inexact \an{first order dual algorithms} for solving the general smooth
convex  problem \eqref{in:eq_prob_princc1}. Since the computation of
the zero-th and \an{the} first order information of the dual problem
\eqref{in:eq_dual_prob} requires the exact solution of the inner
subproblem $\min\limits_{\mathbf{u} \in U} \mathcal{L}(\mathbf{u},\mathbf{x})$ for some fixed
$\mathbf{x} \in \rset^{p}_+$, which generally cannot be computed in practice.
\an{In many practical cases,} inexact dual  information is available by
solving the inner subproblem with a certain inner accuracy.
\an{We} denote with $\tilde{\mathbf{u}}(\mathbf{x})$ the primal point satisfying
the $\delta$-optimality relations:
\begin{equation}\label{in:inexact_inner_opt}
 \tilde{\mathbf{u}}(\mathbf{x}) \in U, \qquad 0 \le \mathcal{L}(\tilde{\mathbf{u}}(\mathbf{x}),\mathbf{x}) - d(\mathbf{x}) \le \delta \qquad \forall \mathbf{x} \in \rset^{p}_+.
\end{equation}

\noindent In relation with \eqref{in:inexact_inner_opt}, we introduce
the following approximations for the dual function and its gradient:
\begin{equation*}
 \tilde{d}(\mathbf{x}) = \mathcal{L}(\tilde{\mathbf{u}}(\mathbf{x}),\mathbf{x}) \quad \text{and} \quad \tilde{\nabla} d(\mathbf{x}) = \mathbf{g}(\tilde{\mathbf{u}}(\mathbf{x})).
\end{equation*}
Then, the following bounds for the dual function $d(\mathbf{x})$ can be
obtained, in terms of a linear and a quadratic model, which use only
approximate information of the dual function and of its gradient
(see \cite[Lemma 2.5]{NecNed:14b}):
\begin{equation}\label{in:inexact_oracle}
 0 \le \left(\tilde{d}(\mathbf{y}) + \langle \tilde{\nabla} d(\mathbf{y}), \mathbf{x}-\mathbf{y} \rangle \right) -
 d(\mathbf{x}) \le L_d\norm{\mathbf{x}-\mathbf{y}}^2 + 3\delta \quad \forall \mathbf{x},\mathbf{y} \in \rset^p_+.
\end{equation}
Note that if $\delta=0$,
then we recover the exact  descent lemma \eqref{in:eq_descent}.
Before we introduce our algorithmic scheme, \an{let us} observe that
\an{one can efficiently solve approximately the inner subproblem} if
the constraint functions  $g_i(\cdot)$ satisfy certain assumptions,
\an{such as either one of the following conditions}:
\begin{itemize}
\item[(1)] \
The operator $\mathbf{g}(\cdot)$ is simple,  i.e.,  given $\mathbf{v} \in U$ and $\mathbf{x} \in \rset^p_+$,
the solution of the following optimization subproblem:
 $$ \min\limits_{\mathbf{u} \in U} \; \left\{\frac{1}{2}\norm{\mathbf{u}-\mathbf{v}}^2 + \langle \mathbf{x}, \mathbf{g}(\mathbf{u})\rangle \right\}$$
 can be obtained in linear time, i.e. $\mathcal{O}(n)$ operations. An example satisfying  this
 assumption is the linear operator, i.e. $\mathbf{g}(\mathbf{u}) = \mathbf{Gu} + \mathbf{g}$,
 where $\mathbf{G} \in \rset^{p \times n}, \mathbf{g} \in \rset^p$.
 \item[(2)] \
 Each function $g_i(\cdot)$ has Lipschitz  \an{continuous gradients}.
\end{itemize}
\an{In such cases, based on Assumption \ref{in:assump_all}(b) (i.e.
$f$ has Lipschitz continuous gradient), it follows that  we can
solve approximately} the inner subproblem $\min\limits_{\mathbf{u}
\in U} \mathcal{L}(\mathbf{u},\mathbf{x})$, for any fixed
$\mathbf{x} \in \rset^{p}_+$, with Nesterov's optimal method for
convex problems with smooth and strongly convex objective function
\cite{Nes:13}. Without \an{loss} of generality, we assume that the
functions $g_i(\cdot)$ are simple. \an{When $\mathbf{g}(\cdot)$
satisfies the above condition (2),} there are minor modifications in
the constants related to the convergence rate. Given $\mathbf{x} \in
\rset^p_+$, the inner approximate optimal point
$\tilde{\mathbf{u}}(\mathbf{x})$ satisfying
$\mathcal{L}(\tilde{\mathbf{u}}(\mathbf{x}),\mathbf{x}) -
d(\mathbf{x}) \le \delta$ is obtained with Nesterov's optimal method
\cite{Nes:13} \an{after $N_\delta$ projections on the simple set $U$
and evaluations of $\nabla f$, where
\begin{equation}\label{in:inner_complexity}
N_\delta= \left \lfloor \sqrt{\frac{L_f}{\sigma_f}} \log \left( \frac{ L_fR_p^2(\bx)}{2\delta} \right)\right \rfloor
\end{equation}
with $R_p(\bx) = \norm{\mathbf{v}^0 - \mathbf{u}(\mathbf{x})}$, and $\mathbf{v}^0$ being the initial  point
of Nesterov's optimal method.} \an{When the
simple feasible set $U$ is compact with a diameter $R_p$ (such as for example in
MPC applications)}, we can bound \an{$R_p(\mathbf{x})$ uniformly, i.e.,
\[R_p(\mathbf{x}) \le R_p\qquad\forall \mathbf{x} \in \rset^p_+.\]
In the sequel, {\it we always assume that such a bound exists}, and
{\it we use warm-start when solving the inner subproblem}.}
Now, we
introduce a general algorithmic scheme, called
Inexact Dual First Order Method (IDFOM), and analyze its convergence
properties, computational complexity and \an{numerical performance}.

\vspace{5pt}

\noindent \textbf{Algorithm IDFOM}

\noindent \begin{tabular}{p{\textwidth}} \hline

\vspace{3pt}

Given $\mathbf{y}^0 \in \rset^p_{+}, \;  \delta>0$, \;  for $k \geq 0$ compute:\\
1. \; Find $\mathbf{u}^k \in U$ such that $\mathcal{L}(\mathbf{u}^k,\mathbf{y}^k) - d(\mathbf{y}^k) \le \delta$, \\
2. \; Update ${\mathbf{x}}^{k}=\left[\mathbf{y}^k+ \frac{1}{2L_{\text{d}}} \tilde{\nabla} d(\mathbf{y}^k)\right]_+$,\\
3. \; Update $\mathbf{y}^{k+1} = \left(1- \theta_k \right)\mathbf{x}^k + \theta_k
\left[\mathbf{y}^0 + \frac{1}{2L_{\text{d}}}\sum\limits_{j=0}^k
\frac{j+1}{2}\tilde{\nabla} d(\mathbf{y}^j) \right]_+$.
\\
\hline
\end{tabular}

\vspace{5pt}

\noindent where $\mathbf{u}^k= \tilde{\mathbf{u}}(\mathbf{y}^k)$,  $\tilde{\nabla} d(\mathbf{y}^k) =
\mathbf{g}(\mathbf{u}^k)$ and \an{the selection of the parameter $\theta_k$ is discussed next}. More precisely, we
distinguish two particular well-known schemes of the above
framework:
\begin{itemize}
\item \textbf{IDGM}: by setting $\theta_k = 0$ \an{for all $k\ge0$}, we recover  the Inexact Dual
Gradient Method since $\mathbf{y}^{k+1} = \mathbf{x}^k$. For this scheme, we define
the dual average sequence $\hat{\mathbf{x}}^k = \frac{1}{k+1}
\sum\limits_{j=0}^k \mathbf{x}^j$. We redefine the dual  final point (the
dual last iterate $\mathbf{x}^k$ when some stopping criterion is satisfied)
as $\mathbf{x}^k = \left[\hat{\mathbf{x}}^k + \frac{1}{2L_{\text{d}}} \tilde{\nabla}
d(\hat{\mathbf{x}}^k) \right]_+$. Thus, all the results concerning $\mathbf{x}^k$
generated by the algorithm  \textbf{IDGM} will refer to this
definition.
\item \textbf{IDFGM}: by setting $\theta_k = \frac{2}{k+3}$ \an{for all $k\ge0$},
we recover the Inexact Dual Fast Gradient Method. This variant has
been analyzed in \cite{Nes:05, DevGli:11,NecNed:14b}.
\end{itemize}

\noindent Note that both dual sequences are dual feasible, i.e.,
$\mathbf{x}^k, \mathbf{y}^k \in \rset^p_+$ \an{for all $k\ge0$}, and thus the inner subproblem
$\min\limits_{\mathbf{u} \in U} \mathcal{L}(\mathbf{u},\mathbf{y}^k)$ has the objective
function strongly convex. Towards estimating the computational
complexity of \textbf{IDFOM}, we present an unified outer
convergence rate for both schemes \textbf{IDGM} and \textbf{IDFGM}
of algorithm \textbf{IDFOM} in terms of dual suboptimality. The
result has been proved in \cite{DevGli:11, NecNed:14b}.

\begin{theorem} \cite{DevGli:11,NecNed:14b}\label{in:th_dual_convrate_idfom}
Given $\delta>0$, let \an{$\{(\mathbf{x}^k,\mathbf{y}^k)\}_{k \ge 0}$} be the dual sequences
generated by algorithm \textbf{IDFOM}. Under Assumption
\ref{in:assump_all}, the following relation holds:
\begin{equation*}
 f^* - d(\mathbf{x}^k) \le \frac{L_{\text{d}} R_d^2}{k^{p(\theta)}} + 4k^{p(\theta)-1}\delta
 \an{\qquad\forall k\ge 1,}
\end{equation*}
where $p(\theta) =
\begin{cases} 1, & \text{if} \;\; \theta_k=0 \\ 2, & \text{if} \;\; \theta_k=\frac{2}{k+3}\end{cases}$ \quad and \quad
$R_d = \min\limits_{\mathbf{x}^* \in X^*}\;\norm{\mathbf{y}^0- \mathbf{x}^*}$.
\end{theorem}

\begin{proof}
Firstly, consider the case $\theta_k = 0$ (which implies $p(\theta)= 1$).
Note that the approximate convexity and Lipschitz continuity
relations \eqref{in:inexact_oracle} lead to:
\begin{align}\label{in:descent_step}
d(\mathbf{x}^{k}) &\ge \tilde{d}(\hat{\mathbf{x}}^{k})  + \langle
\tilde{\nabla} d(\hat{\mathbf{x}}^{k}),
\mathbf{x}^{k}-\hat{\mathbf{x}}^{k}\rangle
- L_{\text{d}}\norm{\mathbf{x}^{k} - \hat{\mathbf{x}}^{k}}^2 - 3\delta \nonumber \\
& \ge \tilde{d}(\hat{\mathbf{x}}^{k}) +
L_{\text{d}}\norm{\mathbf{x}^{k}-\hat{\mathbf{x}}^{k}}^2 -
3\delta\cr & \overset{\eqref{in:inexact_inner_opt}}{\ge}
d(\hat{\mathbf{x}}^{k}) - 3\delta,
\end{align}
where in the second inequality we have used the optimality
conditions of   $\mathbf{x}^k = [\hat{\mathbf{x}}^k + \frac{1}{2L_{\text{d}}}
\tilde{\nabla} d(\hat{\mathbf{x}}^k) ]_+ \in \rset^p_+$. On the other hand,
using \cite[Theorem 2]{DevGli:11}, \an{the following
convergence rate for the dual average point $\hat{\mathbf{x}}^k$ can be derived}:
\begin{equation}\label{in:rate_average}
 f^* - d(\hat{\mathbf{x}}^k) \le \frac{L_{\text{d}} R_d^2}{2k} + \delta
 \an{\qquad\forall k\ge 1}.
\end{equation}
Combining \eqref{in:descent_step} and \eqref{in:rate_average} we \an{obtain} the
first case of the theorem. \an{The second case, concerning $\theta_k =
\frac{2}{k+3}$, has been shown} in
\cite{DevGli:11, NecNed:14b}. \qed
\end{proof}


\noindent Our iteration complexity analysis for algorithm {\bf
IDFOM} is based on two types of approximate primal solutions: the
primal last iterate sequence $(\mathbf{v}^k)_{k \geq 0}$ defined by
$\mathbf{v}^k = \tilde{\mathbf{u}}(\mathbf{x}^k)$ or a primal
average sequence $(\hat{\mathbf{u}}^k)_{k \geq 0}$ of the form:
\begin{align}
\label{in:av_idg}
\hat{\mathbf{u}}^k = \begin{cases} \frac{1}{k+1}\sum\limits_{j=0}^k \mathbf{u}^j, & \text{if} \;\; \textbf{IDGM} \\
  \an{\frac{2}{(k+1)(k+2)}\sum\limits_{j=0}^k (j+1)\mathbf{u}^j}, & \text{if} \;\; \textbf{IDFGM}. \end{cases}
\end{align}
Note that for algorithm \textbf{IDGM} we have $\mathbf{v}^k=\mathbf{u}^k$, while for
algorithm \textbf{IDFGM} $\mathbf{v}^k \not= \mathbf{u}^k$.  Without \an{loss} of
generality, for the simplicity of our results, we assume:
\begin{equation}\label{in:assump_simple}
\mathbf{y}^0 = 0,  \qquad R_d \geq \max \left\{ 1, \frac{1}{c_g},
\frac{\an{L_f} }{c_g} \right\}, \qquad L_{\text{d}} \geq 1.
\end{equation}
If any of these conditions do not hold, then all of the results from
below are valid with minor variations in the constants.



\subsection{Computational complexity of IDFOM in primal last iterate}
\label{in:sec_dfglast} \noindent In this section we derive the
computational complexity for the two main algorithms  \an{within the
framework of} \textbf{IDFOM}, in terms of primal feasibility
violation and primal suboptimality for the last primal iterate
$\mathbf{v}^k = \tilde{\mathbf{u}}(\mathbf{x}^k)$. To obtain these
results, only in this section, we additionally make the following
assumption:

\begin{assumption}\label{in:assump_U_bounded}
The primal set $U$ is compact, i.e. $\max\limits_{\mathbf{u},\mathbf{v} \in U}
\norm{\mathbf{u} - \mathbf{v}} = R_p < \infty$.
\end{assumption}
Assumption \ref{in:assump_U_bounded} implies that the objective
function $f$ is Lipschitz continuous with constant $\bar{L}_f$,
where $\bar{L}_f = \max\limits_{\mathbf{u} \in U} \; \norm{\nabla f(\mathbf{u})}$.
Now, we are ready to prove the main result of this section, \an{given in the following theorem}.

\begin{theorem}\label{in:th_last}
Let $\epsilon> 0$ be some desired accuracy and $\mathbf{v}^k =
\tilde{\mathbf{u}}(\mathbf{x}^k)$ be the primal last iterate
generated by algorithm \textbf{IDFOM}. Under Assumptions
\ref{in:assump_all} and \ref{in:assump_U_bounded}, by setting:
\begin{equation}\label{in:delta_idfom_last}
 \delta \le \frac{L_{\text{d}}R_d^2 }{2\alpha^{p(\theta) - 1}} \left( \frac{\epsilon}{6 L_{\text{d}} R_d^{2}}\right)^{4 - 2/p(\theta)} ,
\end{equation}
where $\alpha = \max \left\{1, \left(\frac{\bar{L}_f}{c_g R_d}
\right)^{2/p(\theta)} \right\}$, the following assertions hold:
\begin{enumerate}
\item[$(i)$] The primal iterate $\mathbf{v}^k$ is $\epsilon$-optimal after
$ \left\lfloor \alpha\left( \frac{6 L_{\text{d}}
R_d^2}{\epsilon}\right)^{2/p(\theta)} \right\rfloor$ outer
iterations.

\item[$(ii)$] Assuming that the primal iterate $\mathbf{v}^k$ is obtained with
\an{Nesterov's optimal method \cite{Nes:13} applied to} the subproblem
$\min\limits_{\mathbf{u} \in U} \; \mathcal{L}(\mathbf{u}, \mathbf{x}^k)$, then $\mathbf{v}^k$ is
$\epsilon$-optimal after
\begin{align*}
\left\lfloor  \! \sqrt{\frac{L_f}{\sigma}} \! \left( \frac{6
L_{\text{d}} R_d^2}{\epsilon}\right)^{\frac{2}{p(\theta)}}
\!\left[\! \left( 4 - \frac{2}{p(\theta)}\right) \log \left( \frac{6
L_{\text{d}} R_d^2}{\epsilon}\right)
  + \log \left( \frac{L_f R_p^2 \alpha^{p(\theta)-1}}{L_{\text{d}} R_d^2} \! \right) \right]  \right\rfloor
\end{align*}
total number of projections on the primal simple  set $U$ and
evaluations of $\nabla f$.
\end{enumerate}
\end{theorem}

\begin{proof}
$(i)$  From Assumption \eqref{in:assump_all}$(b)$, the Lagrangian
$\mathcal{L}(\mathbf{u},\mathbf{x})$ is $\sigma_f$-strongly convex in the variable $\mathbf{u}$
for any fixed $\mathbf{x} \in \rset^p_{+}$, which gives the following
inequality \cite{Nes:04}:
\begin{equation}\label{in:dist_lagr}
\mathcal{L}(\mathbf{u},\mathbf{x}) \ge {d}(\mathbf{x}) + \frac{\sigma_f}{2}\norm{\mathbf{u}(\mathbf{x})-\mathbf{u}}^2
\quad \forall \mathbf{u} \in U, \mathbf{x} \in \rset^p_{+}.
\end{equation}
\an{Moreover, under the strong convexity assumption on $f$ (cf.~Assumption \eqref{in:assump_all}$(b)$), the primal problem
\eqref{in:eq_prob_princc1} has a unique optimal solution, denoted by $\mathbf{u}^*$.}
Using \an{the fact that} $\langle \mathbf{x}, \mathbf{g}(\mathbf{u}^*) \rangle \leq 0$ for any $\mathbf{x} \geq 0$, we have:
\begin{align}
\mathcal{L}(\mathbf{u}^*,\mathbf{x}) - d(\mathbf{x})  &= f(\mathbf{u}^*)+ \langle \mathbf{x}, \mathbf{g}(\mathbf{u}^*) \rangle - d(\mathbf{x})  \leq
f^* - d(\mathbf{x}) \quad  \forall \mathbf{x} \in \rset^p_{+} \label{in:ineq_x}.
\end{align}
Combining \eqref{in:ineq_x} and \eqref{in:dist_lagr} we obtain \an{the following relation
\begin{equation}\label{in:dist_primal_dual}
 \frac{\sigma_f}{2}\norm{\mathbf{u}(\mathbf{x}) - \mathbf{u}^*}^2 \le f^* - d(\mathbf{x}) \qquad \forall \mathbf{x} \in \rset^p_+,
\end{equation}
which provides the distance from $\mathbf{u}(\mathbf{x})$ to the unique optimal solution
$\mathbf{u}^*$.} 

On the other hand, taking $\mathbf{u} = \tilde{\mathbf{u}}(\mathbf{x})$ in \eqref{in:dist_lagr}
and using \eqref{in:inexact_inner_opt}, we have:
\begin{equation}\label{in:dist_inexact_dual}
  \norm{\mathbf{g}(\tilde{\mathbf{u}}(\mathbf{x})) - \mathbf{g}(\mathbf{u}(\mathbf{x}))}
  \le c_g \norm{\mathbf{u}(\mathbf{x}) - \tilde{\mathbf{u}}(\mathbf{x})} \overset{\eqref{in:dist_lagr}}{\le} \sqrt{2 L_{\text{d}} \delta},
\end{equation}
where we used that $L_{\text{d}} = c_g^2/\sigma_f$. From
\eqref{in:dist_primal_dual} and \eqref{in:dist_inexact_dual}, we derive a
link between the primal infeasibility violation and dual
suboptimality gap. Indeed, using the Lipschitz continuity property
of $\mathbf{g}$, we get:
\begin{align*}
\norm{\mathbf{g}(\tilde{\mathbf{u}}(\mathbf{x})) - \mathbf{g}(\mathbf{u}^*)} &\le
\norm{\mathbf{g}(\tilde{\mathbf{u}}(\mathbf{x})) - \mathbf{g}(\mathbf{u}(\mathbf{x})) } + \norm{ \mathbf{g}(\mathbf{u}(\mathbf{x})) -\mathbf{g}(\mathbf{u}^*)} \\
&\overset{\eqref{in:dist_primal_dual} \an{\,\&\,} \eqref{in:dist_inexact_dual}}{\le}
\sqrt{2L_{\text{d}} \delta} + \sqrt{2L_{\text{d}}  (f^* - d(\mathbf{x}))}
\qquad \forall \mathbf{x} \in \rset^p_{+}.
\end{align*}
Combining the above inequality with the property \an{$\mathbf{g}(\mathbf{u}^*) \le \b0$,} and
the fact that for any $\mathbf{a} \in \rset^p$ and $\mathbf{b} \in \rset^p_+$ we have
$\norm{\mathbf{a} + \mathbf{ b} } \ge \norm{[\mathbf{a}]_+}$, we obtain:
\begin{align}\label{in:primal_feas_dual_subopt}
 \left\| \left[ \mathbf{g}(\tilde{\mathbf{u}}(\mathbf{x}))\right]_+ \right\| \le
 \sqrt{2L_{\text{d}} \delta} + \sqrt{2L_{\text{d}}  (f^* - d(\mathbf{x}))}
 \qquad \forall \mathbf{x} \in \rset^p_+.
\end{align}

\noindent Secondly, we find a link between the primal and dual
suboptimality. Indeed, using $\langle \mathbf{x}^*, \mathbf{g}(\mathbf{u}^*)\rangle = 0$, we
have for all $\mathbf{x} \in \rset^p_+$:
\begin{align*}
f^* = \langle \mathbf{x}^*, \mathbf{g}(\mathbf{u}^*) \rangle + f(\mathbf{u}^*)
= \min_{\mathbf{u} \in U} \an{\left \{ f(\mathbf{u}) + \langle \mathbf{x}^*, \mathbf{g}(\mathbf{u}) \rangle\right\} }
\leq f(\tilde{\mathbf{u}}(\mathbf{x})) + \langle \mathbf{x}^*,
\mathbf{g}(\tilde{\mathbf{u}}(\mathbf{x})) \rangle.
\end{align*}
Further, using the \an{Cauchy-Schwarz} inequality, we derive:
\begin{align} \label{in:subopt_primal_dual}
f(\tilde{\mathbf{u}}(\mathbf{x})) - f^*  & \geq - \|\mathbf{x}^*\| \| \mathbf{g}(\mathbf{u}^*) - \mathbf{g}(\tilde{\mathbf{u}}(\mathbf{x})) \| \nonumber \\
& \ge - R_d \left( \sqrt{2L_{\text{d}} \delta} + \sqrt{2L_{\text{d}}
(f^* - d(\mathbf{x}))}\right) \qquad \forall \mathbf{x} \in \rset^p_+.
\end{align}

\noindent On the other hand, from Assumption \ref{in:assump_U_bounded},
we obtain:
\begin{align}\label{in:primal_opt_right}
 f(\an{\tilde{\mathbf{u}} }(\mathbf{x})) - f^* &\le \bar{L}_f \norm{\tilde{\mathbf{u}}(\mathbf{x}) - \mathbf{u}^*}  \le
 \bar{L}_f \left(\norm{\tilde{\mathbf{u}}(\mathbf{x}) - \mathbf{u}(\mathbf{x})} + \norm{\mathbf{u}(\mathbf{x}) - \mathbf{u}^*} \right) \nonumber\\
 & \overset{\eqref{in:dist_primal_dual} \an{\,\&\,} \eqref{in:dist_inexact_dual}} {\le} \bar{L}_f
 \left(\sqrt{\frac{2\delta}{\sigma_f}}  + \sqrt{\frac{2}{\sigma_f}(f^* - d(\mathbf{x})) }  \right).
\end{align}

\noindent Taking $\mathbf{x} = \mathbf{x}^k$ in relation
\eqref{in:primal_feas_dual_subopt} and combining with the dual
convergence rate from Theorem \ref{in:th_dual_convrate_idfom},  we
obtain a convergence estimate on primal infeasibility:
\begin{align*}
\left\| \left[ \mathbf{g} (\mathbf{v}^k)\right]_+ \right\| &\le \frac{2L_{\text{d}}
R_d}{k^{p(\theta)/2}} + \left( 8 L_{\text{d}} k^{p(\theta) - 1}
\delta \right)^{1/2} + (2L_{\text{d}} \delta)^{1/2}.
 \end{align*}

\noindent \an{Letting} $\mathbf{x} = \mathbf{x}^k$ in relations  \eqref{in:subopt_primal_dual}
and \eqref{in:primal_opt_right} and combining with the dual convergence
rate from Theorem \ref{in:th_dual_convrate_idfom},  we obtain
convergence estimates on primal suboptimality:
\begin{align*}
- \frac{2L_{\text{d}} R_d^2 }{k^{p(\theta)/2}} - & \left( 8
L_{\text{d}}R_d^2  k^{p(\theta) - 1} \delta \right)^{1/2}  -
\left( 2L_{\text{d}}R_d^2 \delta \right)^{1/2}  \le f(\mathbf{v}^k) - f^* \\
& \qquad \qquad  \le  \frac{ 2 \bar{L}_f c_g R_d} {\sigma_f
k^{p(\theta)/2}} + \bar{L}_f \left(\frac{8k^{p(\theta)-1} \delta
}{\sigma_f}\right)^{1/2} + \bar{L}_f\left( \frac{2\delta}{\sigma_f}
\right)^{1/2}.
 \end{align*}
Enforcing $\mathbf{v}^k$ to be primal  $\epsilon$-optimal \an{solution in the two
preceding primal convergence rate estimates, we obtain the stated result}.

\noindent $(ii)$ At each outer iteration $k \ge 0$, by combining the
bound \eqref{in:delta_idfom_last} with the inner complexity
\eqref{in:inner_complexity}, Nesterov's optimal method
\cite{Nes:13} for computing $\mathbf{v}^k$ requires:
\begin{equation*}
 \left\lfloor \left( 4 - \frac{2}{p(\theta)}\right)  \sqrt{\frac{L_f}{\sigma}}
 \log \left( \frac{6 L_{\text{d}} R_d^2}{\epsilon}\right)  + \sqrt{\frac{L_f}{\sigma_f}}
 \log \left( \frac{L_f R_p^2}{L_{\text{d}} R_d^2} \alpha^{p(\theta)-1} \right)\right\rfloor
\end{equation*}
projections on the set $U$ and evaluations of $\nabla f$.
Multiplying with the outer complexity given in \an{part $(i)$, we obtain the
result}.   \qed
\end{proof}

\noindent Thus,  we obtained computational complexity estimates for
primal infeasibility and  suboptimality  for the last primal iterate
$\mathbf{v}^k$ of order $\mathcal{O}(\frac{1}{\epsilon^2} \log
\frac{1}{\epsilon})$ for the scheme \textbf{IDGM} and of order
$\mathcal{O}(\frac{1}{\epsilon} \log \frac{1}{\epsilon})$ for the
scheme \textbf{IDFGM}. Furthermore, the inner subproblem needs to be
solved with the inner accuracy $\delta$ of order
$\mathcal{O}(\epsilon^2)$ for  \textbf{IDGM} and of order
$\mathcal{O}(\epsilon^3)$ for  \textbf{IDFGM} in order \an{for the
last primal iterate  $\mathbf{v}^k =
\tilde{\mathbf{u}}(\mathbf{x}^k)$ to be} an $\epsilon$-optimal
primal solution.


\subsection{Computational complexity of IDFOM in primal average iterate}
In this section, we analyze the computational complexity of algorithm
\textbf{IDFOM} in the primal average sequence $\hat{\mathbf{u}}^k$ defined by
\eqref{in:av_idg}. Similar derivations were given in
\cite{NecNed:14b}. For completeness, we also briefly review these
results.   Since the average sequence is different for the two
particular algorithms  \textbf{IDGM} and \textbf{IDFGM}, we provide
separate results. First, we analyze the particular scheme
\textbf{IDGM}, i.e., in \textbf{IDFOM} we choose $\theta_k = 0$ for
all $k \ge 0$. Then, we have  the identity  $\mathbf{y}^{k+1} = \mathbf{x}^k$ and do not assume anymore the
redefinition of the last point
$\mathbf{x}^k = [\hat{\mathbf{x}}^k + \frac{1}{2L_{\text{d}}}\tilde{\nabla} d(\hat{\mathbf{x}}^k)]_+$, i.e.,
algorithm \textbf{IDGM} generates one sequence \an{$\{\mathbf{x}^k\}$} using
the classical gradient update.

\begin{theorem}\label{in:th_idg_av}
Let $\epsilon>0$ and $\mathbf{u}^k = \tilde{\mathbf{u}}(\mathbf{x}^k)$ be the primal sequence
generated by the algorithm \textbf{IDGM} (i.e.  $\theta_k = 0$ for
all $k \ge 0$). Under  Assumption \ref{in:assump_all}, by setting:
\begin{equation}\label{in:delta}
\delta \le \frac{\epsilon}{3}
\end{equation}
the following assertions hold:
\begin{enumerate}
\item[$(i)$] The primal average sequence $\hat{\mathbf{u}}^k$ given in \eqref{in:av_idg}
is $\epsilon$-optimal after   $\left \lfloor \frac{8 L_{\text{d}}
R_d^2}{\epsilon}\right\rfloor$ outer iterations.
\item[$(ii)$] Assuming that the primal   iterate $\mathbf{u}^k = \tilde{\mathbf{u}}(\mathbf{x}^k)$ is obtained
\an{by applying Nesterov's optimal method \cite{Nes:13} to} the
subproblem $\min\limits_{\mathbf{u} \in U} \;
\mathcal{L}(\mathbf{u},\mathbf{x}^k)$, the  primal average iterate
$\hat{\mathbf{u}}^k$ is $\epsilon$-optimal after:
\begin{equation*}
 \left\lfloor 8 \left( \frac{L_f}{\sigma_f}\right)^{1/2} \frac{L_{\text{d}}R_d^2}{\epsilon}
 \log \left( \frac{ L_f
 R_p^2}{\epsilon}\right)\right\rfloor
\end{equation*}
total number of projections on the primal simple set $U$ and
evaluations of $\nabla f$.
\end{enumerate}
\end{theorem}

\begin{proof}
$(i)$ Using the definition of $\mathbf{x}^{k+1}$, we have:
 \begin{equation*}
 \mathbf{x}^{j+1} - \mathbf{x}^j = \left[ \mathbf{x}^j + \frac{1}{2L_{\text{d}}} \tilde{\nabla} \; d(\mathbf{x}^j) \right]_+ - \mathbf{x}^j  \quad \forall j \ge 0.
 \end{equation*}
Summing up the inequalities for $j = 0, \dots, k$ and dividing by
$k$, implies:
\begin{align*}
& \frac{2L_{\text{d}}}{k+1} (\mathbf{x}^{k+1} - \mathbf{x}^0)  =  \frac{2L_{\text{d}}}{k+1}\left(\sum\limits_{j=0}^k
\left[ \mathbf{x}^j + \frac{1}{2L_{\text{d}}} \tilde{\nabla} d(\mathbf{x}^j) \right]_+ - \mathbf{x}^j\right)  \\
& = \! \frac{2L_{\text{d}}}{k+1} \! \left[ \sum\limits_{j=0}^k \!
\left[ \mathbf{x}^j + \frac{1}{2L_{\text{d}}} \tilde{\nabla} d(\mathbf{x}^j) \right]_+
\!-\!  \left(\mathbf{x}^j + \frac{1}{2L_{\text{d}}} \tilde{\nabla} d(\mathbf{x}^j)
\right) \! \right] \!+\! \frac{1}{k+1} \! \sum\limits_{j=0}^k
\tilde{\nabla} d(\mathbf{x}^j).
\end{align*}
Using the fact that $\tilde{\nabla} d(\mathbf{x}^j) = \mathbf{g}(\mathbf{u}^j)$,  the convexity
of $\mathbf{g}$ and denoting $\mathbf{z}^j = \left[\mathbf{x}^j +  \frac{1}{2L_{\text{d}}}
\tilde{\nabla} d(\mathbf{x}^j) \right]_+ - \left(\mathbf{x}^j +
\frac{1}{2L_{\text{d}}} \tilde{\nabla} d(\mathbf{x}^j) \right) \in
\rset^p_+$, we get:
\begin{equation*}
 \mathbf{g}(\hat{\mathbf{u}}^k) + \frac{2L_{\text{d}}}{k+1}\sum\limits_{j=0}^k \mathbf{z}^j
 \le \frac{2L_{\text{d}}}{k+1}(\mathbf{x}^{k+1} - \mathbf{x}^0).
\end{equation*}
Note that if \an{a vector} pair $(\mathbf{a},\mathbf{b})$ satisfies $\mathbf{a} \le \mathbf{b}$, then
$[\mathbf{a}]_+ \le [\mathbf{b}]_+ $ and  $\norm{[\mathbf{a}]_+} \le \norm{[\mathbf{b}]_+}$.
Using \an{these relations and the fact that $\mathbf{z}^j \ge 0 $, we} obtain the following
convergence rate on the feasibility violation:
\begin{align}
\left \| \left[ \mathbf{g}(\hat{\mathbf{u}}^k)\right]_+\right\| & \le \left \|
\left[ \mathbf{g}(\hat{u}^k) + \frac{2L_{\text{d}}}{k+1} \sum\limits_{j=0}^k \mathbf{z}^j \right]_+
\right\| \le \frac{2L_{\text{d}}}{k+1} \left\| \left[\mathbf{x}^{k+1} -
\mathbf{x}^0 \right]_+ \right\| \nonumber \\
& \le \frac{2L_{\text{d}} \norm{\mathbf{x}^{k+1}-\mathbf{x}^0}}{k+1}. \label{in:aux_feas}
\end{align}
On the other hand, from \cite[Theorem 3.1]{NecNed:14b}, it can be
derived that:
\begin{equation}\label{in:res_descent_general}
\norm{\mathbf{x}^{j+1} \!-\! \mathbf{x}}^2 \!\le\! \norm{\mathbf{x}^j \!-\! \mathbf{x}}^2 \!-\!
\frac{1}{L_{\text{d}}} \langle \tilde{\nabla} d(\mathbf{x}^j), \mathbf{x} \!-\!
\mathbf{x}^j\rangle \!+\! \frac{1}{L_{\text{d}}}\!\left(d(\mathbf{x}^{j+1}) \!-\!
\tilde{d}(\mathbf{x}^j) \!+\! 3\delta\right),
\end{equation}
for all $\mathbf{x} \ge 0$ and $j \geq 0$.  Using
\eqref{in:inexact_oracle}, i.e. $d(\mathbf{x}) \leq
\tilde{d}(\mathbf{x}^j) + \langle \tilde{\nabla} d(\mathbf{x}^j),
\mathbf{x} - \mathbf{x}^j\rangle$, taking $\mathbf{x} =
\mathbf{x}^*$, using $d(\mathbf{x}^{j+1}) \leq d(\mathbf{x}^*)$ and
summing \an{over $j$ from $j=0$ to $k$}, we obtain:
\begin{equation}
\label{in:res_descent} \norm{\mathbf{x}^{k+1} - \mathbf{x}^*} \le \norm{\mathbf{x}^0  - \mathbf{x}^*} +
\sqrt{\frac{3 \delta(k+1)}{L_{\text{d}}}}.
\end{equation}

\noindent Combining the estimate for feasibility violation
\eqref{in:aux_feas} and \eqref{in:res_descent}, we finally have:
\begin{equation}
\label{in:feas_rate_final}
 \!\left\| \left[ \mathbf{g}(\hat{\mathbf{u}}^k)\right]_+ \right\| \!\le\!
 \frac{2 L_{\text{d}} (\norm{\mathbf{x}^0 \!-\! \mathbf{x}^*} \!+\! \norm{\mathbf{x}^{k+1} \!-\! \mathbf{x}^*})}{k+1}
 \!\le\! \frac{4L_{\text{d}} \norm{\mathbf{x}^0 \!-\! \mathbf{x}^*}}{k+1} \!+\! 2\sqrt{\frac{3L_{\text{d}} \delta}{k+1}}.
\end{equation}

\noindent In order to obtain a sublinear estimate on the primal
suboptimality, we \an{write}:
\begin{align}
 f^* & = \min\limits_{\mathbf{u} \in U} \an{\{f(\mathbf{u}) + \langle \mathbf{x}^*, \mathbf{g}(\mathbf{u}) \rangle\} }
 \le f(\hat{\mathbf{u}}^k) + \langle \mathbf{x}^*, \mathbf{g}(\hat{\mathbf{u}}^k)\rangle
 \le f(\hat{\mathbf{u}}^k) + \langle \mathbf{x}^*, \left[ \mathbf{g}(\hat{\mathbf{u}}^k)\right]_+\rangle \nonumber \\
 & \le f(\hat{\mathbf{u}}^k) + \norm{\mathbf{x}^*} \left\| \left[\mathbf{g}(\hat{\mathbf{u}}^k) \right]_+ \right\|
 \le f(\hat{\mathbf{u}}^k) + (R_d + \norm{\mathbf{x}^0}) \left\| \left[ \mathbf{g}(\hat{\mathbf{u}}^k)\right]_+ \right\|.\label{in:subopt_left}
\end{align}
On the other hand, taking $\mathbf{x} = \b0$ in \eqref{in:res_descent_general} and
using the definition of $\tilde{d}(\mathbf{x}^j)$, we obtain:
\begin{align*}
\norm{\mathbf{x}^{j+1}}^2  &\le \norm{\mathbf{x}^j}^2 \!+\! \frac{1}{L_{\text{d}}}
\langle \tilde{\nabla} d(\mathbf{x}^j), \mathbf{x}^j \rangle \!+\!
\frac{1}{L_{\text{d}}}
\! \left( \!d(\mathbf{x}^{j+1}) \!-\! f(\mathbf{u}^j) - \langle \mathbf{x}^j, \tilde{\nabla} d(\mathbf{x}^j)\rangle  \!+\! 3\delta \right)\\
& \le \norm{\mathbf{x}^j}^2 \!+\! \frac{1}{L_d} \left( f^* \!-\! f(\mathbf{u}^j) \!+\!
3\delta\right).
\end{align*}
Using an inductive argument, the convexity of $f$ and the definition
of $\hat{\mathbf{u}}^k$, we get:
\begin{equation}\label{in:subopt_right}
 f(\hat{\mathbf{u}}^k) - f^* \le \frac{L_{\text{d}} \norm{\mathbf{x}^0}^2}{k+1} + 3\delta.
\end{equation}
Using the assumption $\mathbf{x}^0 = \b0$, from \eqref{in:feas_rate_final},
\eqref{in:subopt_left} and \eqref{in:subopt_right}, we get:
\begin{equation*}
- \frac{4L_{\text{d}} R_d^2}{k+1} - 2R_d\sqrt{\frac{3L_{\text{d}}
\delta}{k+1}} \le f(\an{\hat{\mathbf{u}}^k}) - f^* \le 3 \delta.
\end{equation*}
From assumptions \an{on the constants $R_d$, $L_{\text{d}}$ and $\delta$
(see \eqref{in:assump_simple} and \eqref{in:delta}), our first result follows.}

\noindent $(ii)$ Taking into account the relation \eqref{in:delta} on
$\delta$, the inner number of  projections on the simple set $U$ at
each outer iteration is given by:
\[ \left\lfloor \left( \frac{L_f}{\sigma_f}\right)^{1/2} \log
\left( \frac{ L_f
 R_p^2}{\epsilon}\right)\right\rfloor. \]
Multiplying with the outer complexity obtained in $(i)$, we get the
second result. \qed
\end{proof}

\noindent Further, we study the computational complexity of the
second particular algorithm \textbf{IDFGM}, i.e. the scheme
\textbf{IDFOM} with $\theta_k = \frac{2}{k+3}$. Note that in the
framework \textbf{IDFOM} both sequences \an{$\{\mathbf{x}^k\}_{k \ge 0}$ and
$\{\mathbf{y}^k\}_{k \ge 0}$} are dual feasible, i.e. are in $\rset^p_+$.  Based
on \cite[Theorem 2]{Nes:05} (see also \cite{DevGli:11,NecSuy:08}), when
$\theta_k =  \frac{2}{k+3}$, we have the following inequality which
will help us to establish the convergence properties of the
particular algorithm \textbf{IDFGM}:
\begin{align}
\label{in:nesterov_relation}
& \frac{(k+1)(k+2)}{4}   d(\mathbf{x}^k)  + \frac{(k+1)(k+2)(k+3)}{4}\delta   \\
& \qquad \ge \max_{\an{\mathbf{x} \ge \b0} } \left( -L_{\text{d}}\norm{\mathbf{x}-\mathbf{y}^0}^2  +
\sum\limits_{j=0}^k \frac{j+1}{2}\left[ \tilde{d}(\mathbf{y}^j) + \langle
\tilde{\nabla} d(\mathbf{y}^j), \mathbf{x} - \mathbf{y}^j\rangle \right] \right).  \nonumber
\end{align}

\noindent We now derive complexity estimates for primal
infeasibility and suboptimality of the average primal sequence
\an{$\{\hat{\mathbf{u}}^k\}_{k \ge 0}$} as defined in \eqref{in:av_idg} for algorithm
\textbf{IDFGM}.

\begin{theorem}\label{in:th_idfg_av}
Let $\epsilon > 0$ and $\mathbf{u}^k = \tilde{\mathbf{u}}(\mathbf{y}^k)$ be the primal sequence
generated by  algorithm \textbf{IDFGM} (i.e.   $\theta_k =
\frac{2}{k+3}$  for all $k \geq 0$). Under Assumption
\ref{in:assump_all}, by setting:
\begin{equation}\label{in:delta_bound_idfg_av}
\delta \le \frac{\epsilon^{3/2}}{8 L_{\text{d}}^{1/2} R_d},
\end{equation}
the following assertions hold:

\noindent $(i)$ The primal average iterate  $\hat{\mathbf{u}}^k$ given in
\eqref{in:av_idg} is $\epsilon$-optimal after $\left \lfloor
\left(\frac{32L_{\text{d}} R_d^2}{\epsilon}
\right)^{1/2}\right\rfloor$ outer iterations.

\noindent $(ii)$ Assuming that the primal iterate $\mathbf{u}^k =
\tilde{\mathbf{u}}(\mathbf{y}^k)$ is obtained \an{by applying Nesterov's optimal method
\cite{Nes:13} to the subproblem $\min\limits_{\mathbf{u} \in U} \;
\mathcal{L}(\mathbf{u},\mathbf{y}^k)$, the} average primal iterate $\hat{\mathbf{u}}^k$ is
$\epsilon$-optimal after:
 \begin{equation*}
 \left \lfloor \sqrt{\frac{L_f}{\sigma_f}} \left( \frac{32 L_{\text{d}} R_d^2}{\epsilon}\right)^{1/2}
 \log \left( \frac{4 L_{\text{d}}^{1/2}L_f R_p^2 R_d}{\epsilon^{3/2}} \right) \right \rfloor
 \end{equation*}
total number of projections on the primal simple set $U$ and
evaluations of $\nabla f$.
\end{theorem}

\begin{proof}
\noindent $(i)$ For  primal feasibility estimate, we use
\eqref{in:nesterov_relation} and the convexity of $f$ and $g$:
\begin{align}
\max\limits_{\an{\mathbf{x} \ge \b0}} \an{\left(-\frac{4L_{\text{d}}}{(k+1)^2}\norm{\mathbf{x}-\mathbf{y}^0}^2 +
\langle \mathbf{x}, \mathbf{g}(\hat{\mathbf{u}}^k)\rangle\right)}
\le d(\mathbf{x}^k) - f(\hat{\mathbf{u}}^k) +
(k+3)\delta. \label{in:dual_nes_relation}
\end{align}
For the right hand side term, using \an{the strong duality and $\mathbf{x}^* \ge \b0$},  we have:
\begin{align}
 d(\mathbf{x}^k) - f(\hat{\mathbf{u}}^k)
 &\le d(\mathbf{x}^*) - f(\hat{\mathbf{u}}^k)
 =\min\limits_{\mathbf{u} \in U} \an{ \{f(u) + \langle \mathbf{x}^*, \mathbf{g}(\mathbf{u})\rangle\} }
    - f(\hat{\mathbf{u}}^k) \nonumber\\
& \le \langle \mathbf{x}^*, \mathbf{g}(\hat{\mathbf{u}}^k)\rangle \le \left \langle \mathbf{x}^*,
[\mathbf{g}(\hat{\mathbf{u}}^k)]_+ \right \rangle. \label{in:dual_nes_relation2}
\end{align}
By evaluating the left hand side term in \eqref{in:dual_nes_relation}
at $\mathbf{x} = \frac{(k+1)^2}{8 L_{\text{d}}} [\mathbf{g}(\hat{\mathbf{u}}^k)]_+$ and
observing that $\langle [\mathbf{g}(\hat{\mathbf{u}}^k)]_+, \mathbf{g}(\hat{\mathbf{u}}^k) -
[\mathbf{g}(\hat{\mathbf{u}}^k)]_+\rangle = 0$ we obtain:
\begin{align}
&\max\limits_{\an{\mathbf{x} \ge \b0} } \an{\left( - \frac{4L_{\text{d}}}{(k+1)^2}\norm{\mathbf{x}-\mathbf{y}^0}^2 +
\langle \mathbf{x}, \mathbf{g}(\hat{\mathbf{u}}^k)\rangle\right)} \label{in:dual_nes_relation3}\\
& \ge \frac{(k+1)^2}{16 L_{\text{d}}}\norm{\left[\mathbf{g}(\hat{\mathbf{u}}^k)
\right]_+}^2 \!-\! \frac{4L_{\text{d}}\norm{\mathbf{y}^0}^2}{(k+1)^2} \!+\!
\langle \mathbf{y}^0, \left[ \mathbf{g}(\hat{\mathbf{u}}^k)\right]_+\rangle. \nonumber
\end{align}
Combining \eqref{in:dual_nes_relation2} and \eqref{in:dual_nes_relation3}
with \eqref{in:dual_nes_relation}, using the \an{Cauchy-Schwarz} inequality
and notation $\gamma = \norm{[\mathbf{g}(\hat{\mathbf{u}}^k)]_+}$ we obtain:
\begin{equation*}
 \frac{(k+1)^2}{16 L_{\text{d}}}\gamma^2 - (k+3)\delta - \norm{\mathbf{x}^* - \mathbf{y}^0}\gamma -
 \frac{4L_{\text{d}}\norm{\mathbf{y}^0}^2}{(k+1)^2} \le 0.
\end{equation*}
Thus, $\gamma$ must be less than the largest root of the
second-order equation, from which, together with the definition of
$R_{\text{d}}$ we get:
 \begin{equation}\label{in:infes_av}
 \norm{ \left [ \mathbf{g}(\hat{\mathbf{u}}^k) \right]_+} \le
 \frac{16L_{\text{d}} R_{\text{d}} }{(k+1)^2} +
  4\sqrt{\frac{3 L_{\text{d}} \delta}{k+1}}.
 \end{equation}

\noindent For the left hand side on primal suboptimality, using \an{$\mathbf{x}^*\ge \b0$}, we have:
\begin{align*}
f^* &= \min\limits_{\mathbf{u} \in U} \an{\{ f(\mathbf{u}) + \langle \mathbf{x}^*, \mathbf{g}(\mathbf{u})\rangle\} }
\le f(\hat{\mathbf{u}}^k) + \langle \mathbf{x}^*, \mathbf{g}(\hat{\mathbf{u}}^k)\rangle \\
 & \le f(\hat{\mathbf{u}}^k) + \langle \mathbf{x}^*, [ \mathbf{g}(\hat{\mathbf{u}}^k)]_+\rangle \le
 f(\hat{\mathbf{u}}^k) + R_{\text{d}} \norm{[\mathbf{g}(\hat{\mathbf{u}}^k)]_+}.
\end{align*}
Using  \eqref{in:infes_av}, we derive an estimate on the left hand side
primal suboptimality:
\begin{equation}\label{in:subopt_idfg_left}
 f(\hat{\mathbf{u}}^k) - f^*  \le \frac{16L_{\text{d}} R_d^2}{(k+1)^2} + 4R_d
\sqrt{\frac{3 L_{\text{d}} \delta}{k+1}}.
\end{equation}
On the other hand, taking $\mathbf{x}=\b0$ in \eqref{in:dual_nes_relation} and
\an{recalling} that $\mathbf{y}^0=\b0$, we get:
\begin{align}\label{in:subopt_idfg_right}
 f(\hat{\mathbf{u}}^k) - d(\mathbf{x}^k)
 &\le - \max\limits_{\an{\mathbf{x} \ge \b0}} \an{\left( - \frac{4L_{\text{d}}}{(k+1)^2}\norm{\mathbf{x}-\mathbf{y}^0}^2
+ \langle \mathbf{x}, \mathbf{g}(\hat{\mathbf{u}}^k)\rangle\right)} + (k+3)\delta \nonumber\\
& \le (k+3)\delta.
\end{align}
\an{Moreover, taking into account that $d(\mathbf{x}^k) \le f^*$, from
\eqref{in:subopt_idfg_left} and \eqref{in:subopt_idfg_right} we obtain}:
\begin{equation} \label{in:ps_dfg}
- \frac{16L_{\text{d}} R_d^2}{(k+1)^2} - 4R_d
\sqrt{\frac{3L_{\text{d}} \delta}{k+1}} \le f(\hat{\mathbf{u}}^k) - f^* \le
(k+3)\delta.
\end{equation}
From the convergence rates \eqref{in:infes_av} and \eqref{in:ps_dfg} we
obtain our first result.

\noindent $(ii)$ \an{Substitution of} the bound \eqref{in:delta_bound_idfg_av}
into the inner complexity estimate \eqref{in:inner_complexity} leads
to:
\begin{equation*}
 \left \lfloor \sqrt{\frac{L_f}{\sigma_f}}
 \log \left( \frac{ 4 L_{\text{d}} L_f R_d R_p^2}{\epsilon} \right) \right\rfloor
\end{equation*}
projections on $U$ and evaluations of $\nabla f$ for each outer
iteration. Multiplying with the outer complexity \an{estimate obtained in part $(i)$,} we get
our second result. \qed
\end{proof}

\noindent Thus,  we obtained computational complexity estimates for
primal infeasibility and  suboptimality  for the average of primal
iterates $\hat{\mathbf{u}}^k$  of order
$\mathcal{O}(\frac{1}{\epsilon} \log \frac{1}{\epsilon})$ for the
scheme \textbf{IDGM} and of order
$\mathcal{O}(\frac{1}{\sqrt{\epsilon}} \log \frac{1}{\epsilon})$ for
the scheme \textbf{IDFGM}. Moreover, the inner subproblem needs to
be solved with the inner accuracy $\delta$ of order
$\mathcal{O}(\epsilon)$ for  \textbf{IDGM} and of order
$\mathcal{O}(\epsilon \sqrt{\epsilon})$ for  \textbf{IDFGM} \an{so
that to have the primal average sequence  $\hat{\mathbf{u}}^k$ as}
an $\epsilon$-optimal primal solution. Further, the iteration
complexity estimates  in the last primal  iterate $\mathbf{v}^k$ are
inferior to those estimates corresponding to an average of primal
iterates $\hat{\mathbf{u}}^k$. However, in practical applications we
have observed that algorithm \textbf{IDFOM} converges faster in the
last primal iterate  than in the primal average sequence. Note that
this does not mean that our analysis is weak, since we can also
construct problems which show the behavior predicted by the theory.


\section{DuQuad toolbox}
\noindent In this section, we present the open-source  solver DuQuad
\cite{KwaNec:14} based on C-language implementations of the
framework \textbf{IDFOM} for solving quadratic programs (QP) that
appear in many applications. For example  linear MPC problems are
usually formulated as QPs that need to be solved at each time
instant for a given state. Thus, in this toolbox we considered
convex quadratic programs of the form:
\begin{align}\label{in:original_primal}
\min\limits_{\mathbf{u} \in U}& \; f(\mathbf{u}) \quad \left(: = \frac{1}{2}\mathbf{u}^T\mathbf{Q} \mathbf{u} +
\mathbf{q}^T \mathbf{u}\right)  \qquad  \text{s.t.}:  \quad \mathbf{Gu} + \mathbf{g} \le 0,
\end{align}
where $ \mathbf{Q} \succ 0$, $ \mathbf{G}  \in \rset^{p \times n}$ and $U \subseteq
\rset^n$ is a simple compact  convex set, i.e. a box $ U= [\mathbf{lb}
\; \mathbf{ub}]$. Note that our formulation allows to incorporate in
the QP either linear inequality constraints (arising e.g. in sparse
formulation of predictive control and network utility maximization)
or linear equality constraints (arising  e.g. in condensed
formulation of predictive control and DC optimal power flow). In
fact the user can define linear constraints of the form:
$\bar{\mathbf{lb}} \leq \bar{\mathbf{G}}\mathbf{u}+ \bar{\mathbf{g}} \leq \bar{\mathbf{ub}}$ and
depending on the values for $\bar{\mathbf{lb}}$ and $\bar{\mathbf{ub}}$
we have linear inequalities or equalities. Note that the objective
function of \eqref{in:original_primal} has Lipschitz gradient with
constant $L_{\text{f}} = \lambda_{\max}(\mathbf{Q})$ and  its dual has also
Lipschitz gradient with constant $L_{\text{d}} =
\frac{\norm{\mathbf{G}}^2}{\lambda_{\min}(\mathbf{Q})}$. Based on the scheme
\textbf{IDFOM}, the main iteration in DuQuad consists of two steps:

\vspace{0.1cm}

\noindent \textbf{Step 1}: for a given inner accuracy $\delta>0$ and
a multiplier $\mathbf{x} \in \rset^p_+$, we solve approximately  the inner
problem with accuracy $\delta$ to obtain an approximate solution
$\tilde{\mathbf{u}}(\mathbf{x})$ instead of the exact solution $\mathbf{u}(\mathbf{x})$, i.e.:
$\mathcal{L}(\tilde{\mathbf{u}}(\mathbf{x}),\mathbf{x}) - d(\mathbf{x}) \leq \delta$. In DuQuad, we
obtain an approximate solution $\tilde{\mathbf{u}}(\mathbf{x})$ using Nesterov's
optimal method \cite{Nes:13} and warm-start.

\vspace{0.1cm}

\noindent \textbf{Step 2}: Once a $\delta$-solution  $\tilde{\mathbf{u}}(\mathbf{x})$
for inner subproblem was found, we update at the outer stage the
Lagrange multipliers using the scheme {\bf IDFOM}, i.e. for updating
the Lagrange multipliers we use instead of the true value of the
dual gradient $\nabla d(\mathbf{x}) = \mathbf{G} \mathbf{u}(\mathbf{x}) + \mathbf{g}$, an approximate value
$\tilde{\nabla} d(\mathbf{x}) = \mathbf{G} \tilde{\mathbf{u}}(\mathbf{x}) + \mathbf{g}$.

\begin{figure}[h!]
\begin{center}
\includegraphics[width=0.8\textwidth, height=7cm]{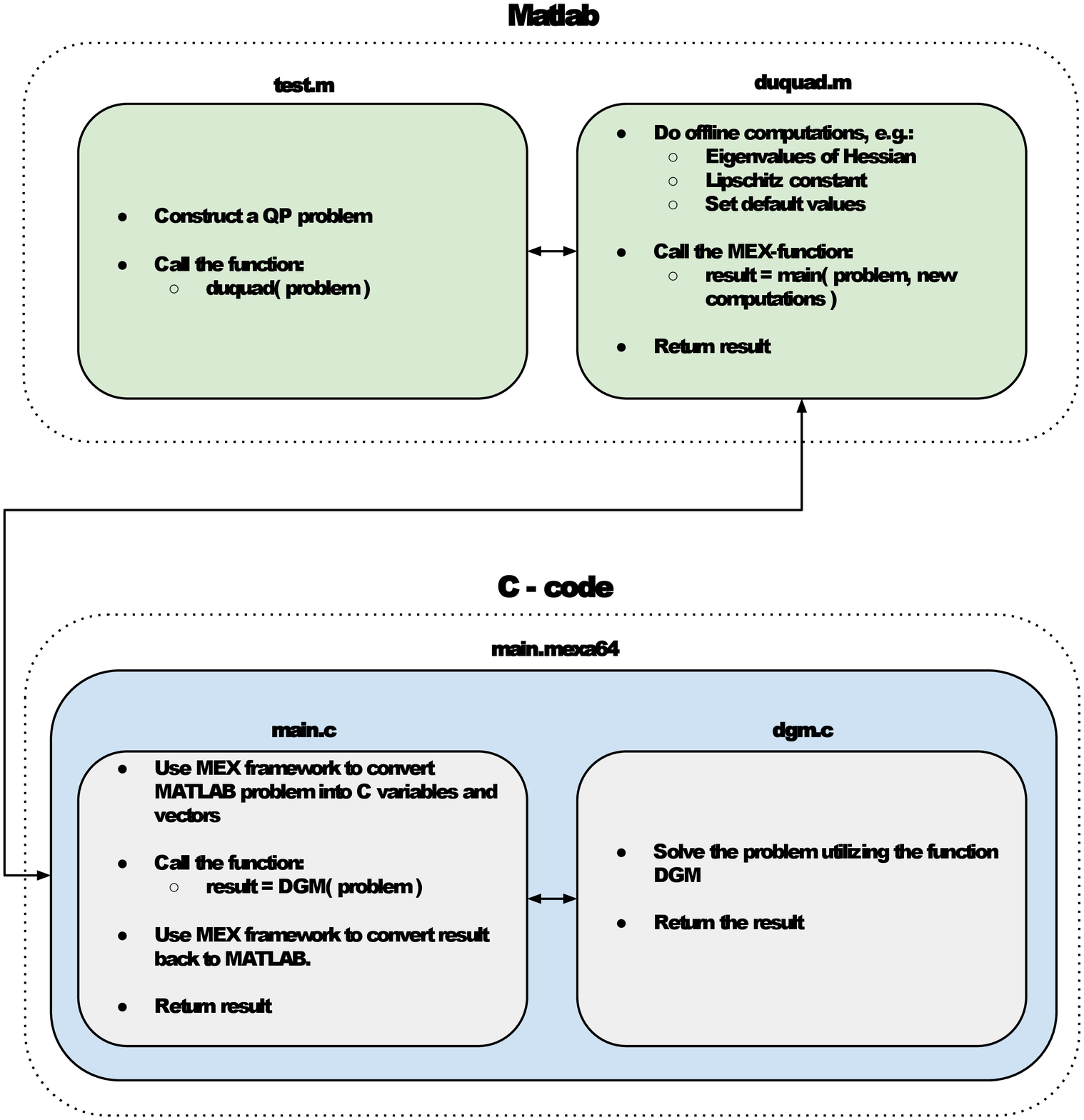}
\caption{DuQuad workflow.}
\label{in:fig:duquad_workflow}
\end{center}
\end{figure}

\vspace{0.1cm}
\noindent An overview of the workflow in DuQuad \cite{KwaNec:14} is illustrated in
Fig. \ref{in:fig:duquad_workflow}. A QP problem is constructed using a
Matlab script called \textit{test.m}. Then, the function
\textit{duquad.m} is called with the problem data as input and \an{it} is
regarded as a preprocessing stage for the online optimization. The
binary MEX file is called, with the original problem data and the
extra \an{information} as \an{an} input. The \textit{main.c} file of the C-code includes
the MEX framework and is able to convert the MATLAB data into C
format. Furthermore, the converted data gets bundled into a C
``struct'' and passed as \an{an} input to the algorithm that solves the
problem using the \an{two steps as described above}.

\noindent  In DuQuad \an{a} user can choose either algorithm
\textbf{IDFGM} or algorithm \textbf{IDGM} for solving the dual
problem.  \an{Moreover, the user} can also choose the inner accuracy $\delta$ for
solving the inner problem. \an{In} the toolbox the default values for
$\delta$ are taken as in Theorems \ref{in:th_last}, \ref{in:th_idg_av} and
\ref{in:th_idfg_av}. From these theorems we conclude that the inner QP
has to be solved with higher accuracy in dual fast gradient
algorithm \textbf{IDFGM} than in dual  gradient algorithm
\textbf{IDGM}. This shows that dual gradient algorithm  \textbf{IDGM}
is  robust to inexact information, while dual fast gradient
algorithm \textbf{IDFGM} is sensitive  to inexact computations, as we can also see from
\an{plots in} Fig. \ref{in:dfgm_sensitivity}.
\begin{figure}[h!]
\centering
\includegraphics[width=0.5\textwidth,height=5cm]{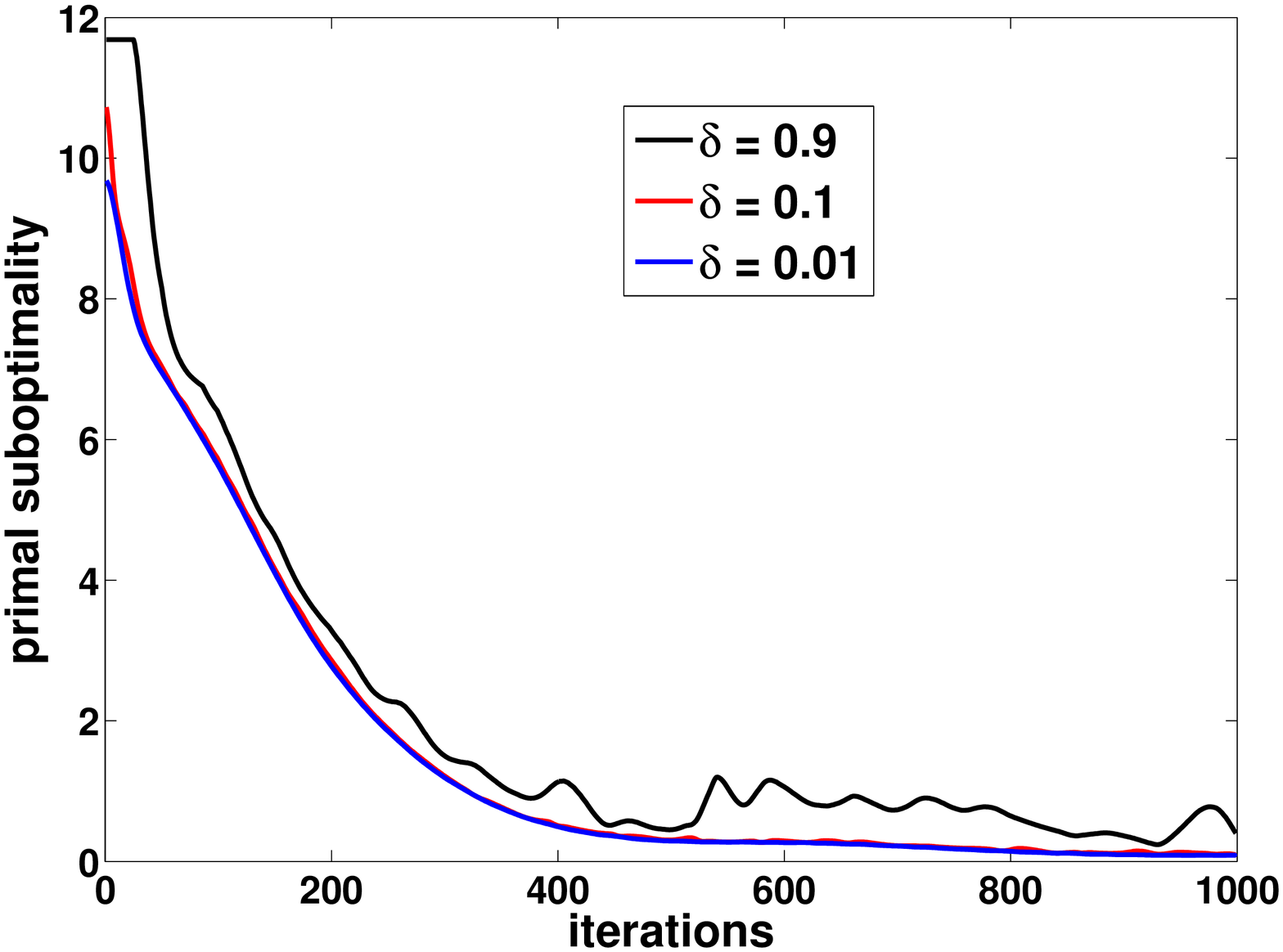}
\hspace{-10pt}
\includegraphics[width=0.51\textwidth,height=5cm]{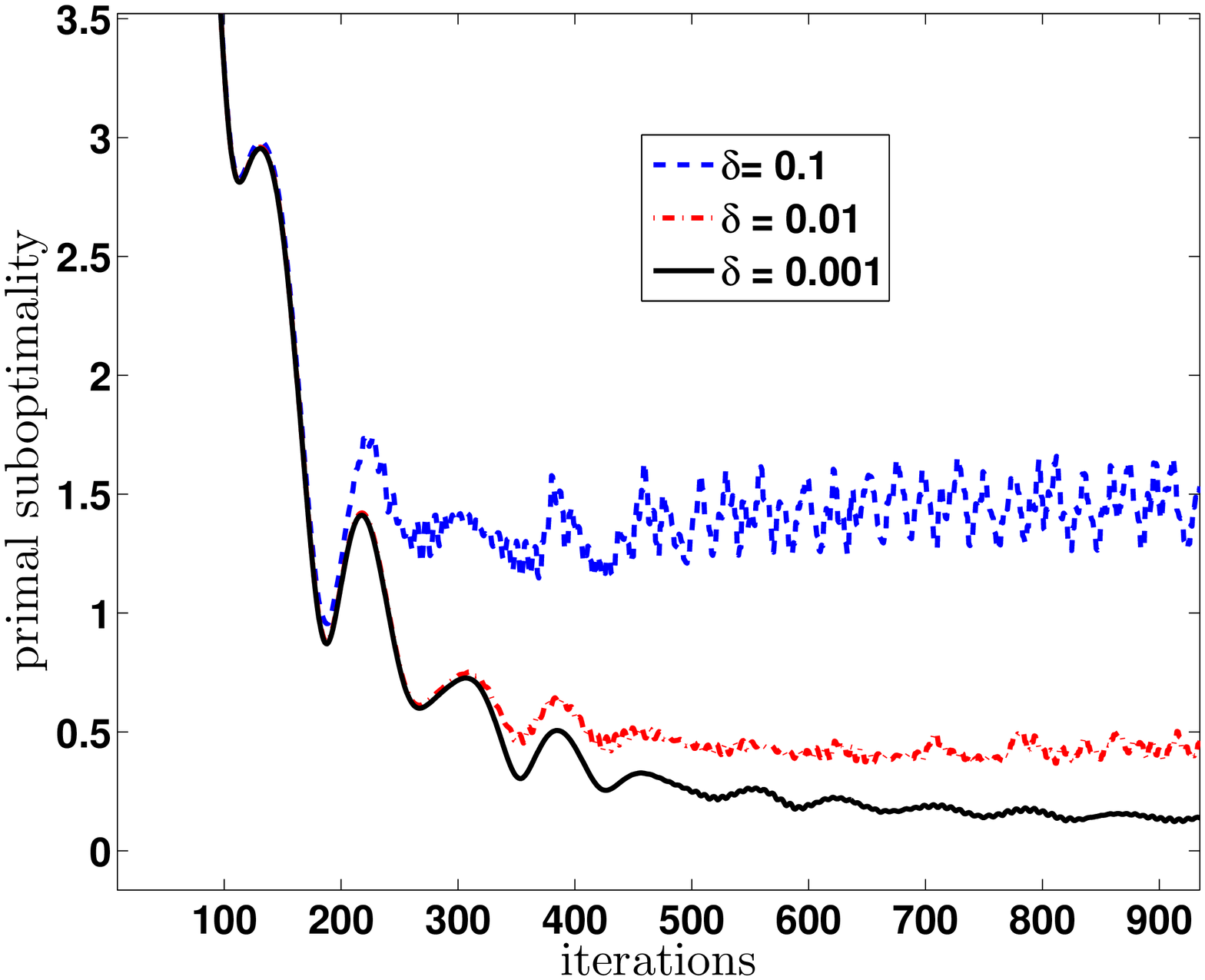}
\caption{Sensitivity  of \textbf{IDGM} (left) and \textbf{IDFGM} (right) in the average of iterates  in terms of primal suboptimality
w.r.t. different values of the inner accuracy $\delta$ for a QP ($n=50$ and $p=75$)
with desired accuracy  $\epsilon=0.01$.} \label{in:dfgm_sensitivity}
\end{figure}

\noindent Let us analyze now the  computational cost per inner and
outer iteration for algorithm {\bf IDFOM} for solving approximately
the original QP problem \eqref{in:original_primal}:

\vspace{2pt}

\noindent \textbf{Inner iteration}: When solving the inner problem
with Nesterov's optimal method \cite{Nes:13}, the main computational
effort is done in computing the gradient of the Lagrangian $ \nabla
\mathcal{L}(\mathbf{u},\mathbf{x}) = \mathbf{Q}\mathbf{u} + \mathbf{q} + \mathbf{G}^T \mathbf{x}$.
In DuQuad these matrix-vector
operations are implemented efficiently  in C (\an{the} matrices that do not
change along iterations are computed once and only $\mathbf{G}^T \mathbf{x}$ is
computed at each outer iteration). The cost for computing $\nabla
\mathcal{L}(\mathbf{u},\mathbf{x})$ for general QPs is ${\cal O} (n^2)$. However, when the
matrices $\mathbf{Q}$  and $\mathbf{G}$ are sparse (e.g. network utility maximization
problem) the cost ${\cal O} (n^2)$ can be reduced substantially. The
other operations in algorithm {\bf IDFOM} are just vector operations
and, \an{hence,} they are of order ${\cal O} (n)$. Thus, the dominant
operation at the inner stage is the matrix-vector product.

\vspace{2pt}

\noindent \textbf{Outer iteration}:
The main
computational effort in the outer iteration of {\bf IDFOM} is done in computing
the inexact gradient of
the dual function:
$ \tilde{\nabla} d(\mathbf{x}) = \mathbf{G} \tilde{\mathbf{u}}(\mathbf{x}) + \mathbf{g}.  $
The cost for computing $\tilde{\nabla} d (\mathbf{x})$ for general QPs is
${\cal O} (np)$. However, when the matrix  $\mathbf{G}$ is sparse, this
cost  can be reduced.  The other operations in algorithm {\bf
IDFOM} are of order ${\cal O} (p)$. \an{Hence}, the
dominant operation  at the outer stage is also the matrix-vector
product.

\vspace{2pt}

\noindent Fig. \ref{in:fig:gprof_n150_dfgm_case1} displays the result
of profiling the code with gprof. In this simulation, a standard QP
with inequality constraints, and \an{with} dimensions $n = 150$ and $p = 225$
was solved by algorithm \textbf{IDFGM}. The profiling summary is
listed in the order of the time spent in each file.  This figure
shows that most of the execution time of the program is spent on
the  library module \textit{math-functions.c}. More exactly, the dominating function
is  \textit{mtx-vec-mul}, which multiplies a matrix with a vector.
\begin{figure}[h!]
\centering
\includegraphics[width=0.6\textwidth,height=6cm]{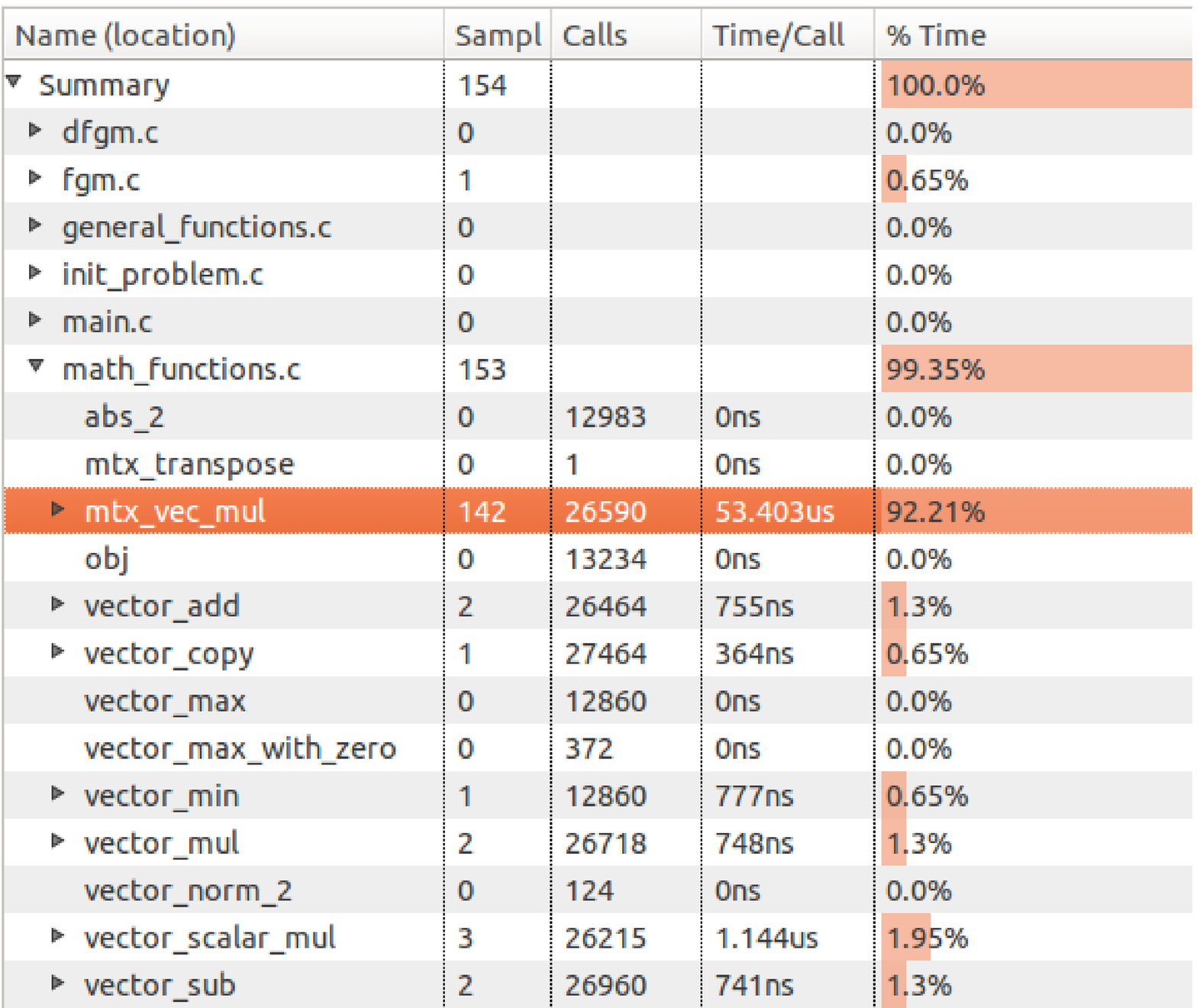}
\caption{Profiling the code with gprof.}
\label{in:fig:gprof_n150_dfgm_case1}
\end{figure}

\noindent In conclusion, in DuQuad the main operations are the
matrix-vector products. Therefore, DuQuad is adequate for solving QP
problems on hardware  with limited resources and capabilities, since
it does not require any solver for linear systems or other
complicating operations, while  most of the existing solvers for QPs
from the literature \an{(such as those implementing active set or interior point
methods)} require the capability of solving linear systems. On the
other hand, DuQuad can be also used for solving large-scale sparse
QP problems \an{since, in this case, the iterations are computationally inexpensive} (only
sparse matrix-vector products).


\section{Numerical simulations with DuQuad}
For  numerical experiments,  using  the solver DuQuade \cite{KwaNec:14}, we \an{at first} consider
random QP problems  and then a real-time MPC controller
for a  self balancing robot.

\subsection{Random QPs}
\noindent In this section we analyze the behavior of the dual first order methods presented in this chapter and implemented in DuQuad for solving random QPs.

\begin{figure}[ht!]
\begin{center}
\vskip-0.1cm
\includegraphics[width=1.05\textwidth,height=5cm]{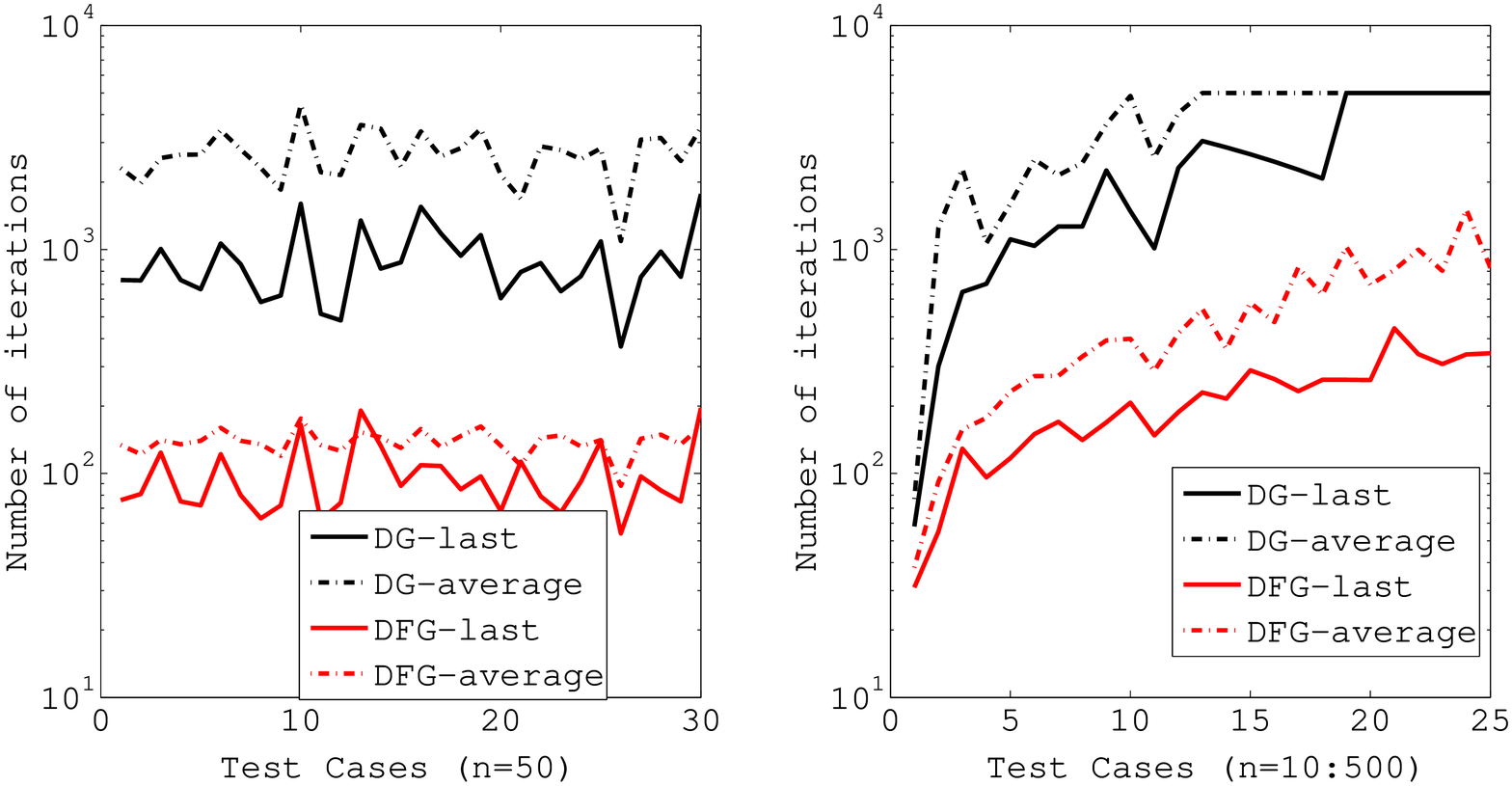}
\caption{Number of outer iterations on random QPs for  \textbf{IDGM}
and  \textbf{IDFGM} in primal last/average of iterates for different
test cases of the same dimension (left) and  of variable
dimension~(right). } \label{in:fig:comparison_dfo}
\end{center}
\end{figure}

\noindent  In Fig.~\ref{in:fig:comparison_dfo} we plot the practical
number of outer iterations on random QPs of algorithms \textbf{IDGM}
and \textbf{IDFGM} for different test cases of the same dimension $n=50$
(left) and for different test cases of variable dimension ranging from $n=10$ to
$n=500$ (right). We \an{have choosen the accuracy $\epsilon=0.01$ and
the stopping criteria is the requirement that both quantities
\[ \abs{f(\mathbf{u}) - f^{*}}  \quad \text{and} \quad \norm{[\mathbf{G} \mathbf{u} + \mathbf{g}]_+} \]
are} less than the accuracy $\epsilon$, \an{where $f^*$ has been computed a priori
with Matlab quadprog}. From this figure we observe that the number of
iterations \an{is} not varying much for different test cases and, also,
that the number of iterations \an{is} mildly dependent on \an{the} problem's
dimension. Finally, we  observe that dual first order  methods perform
usually better in the primal last iterate than in the average of primal  iterates.


\subsection{Real-time MPC for balancing robot}
\noindent In this section we use the dual first order methods presented in this chapter and implemented in DuQuad for solving a real-time MPC control problem.

\begin{figure}[htb!]
\centering
\includegraphics[width=0.5\textwidth,,height=5cm]{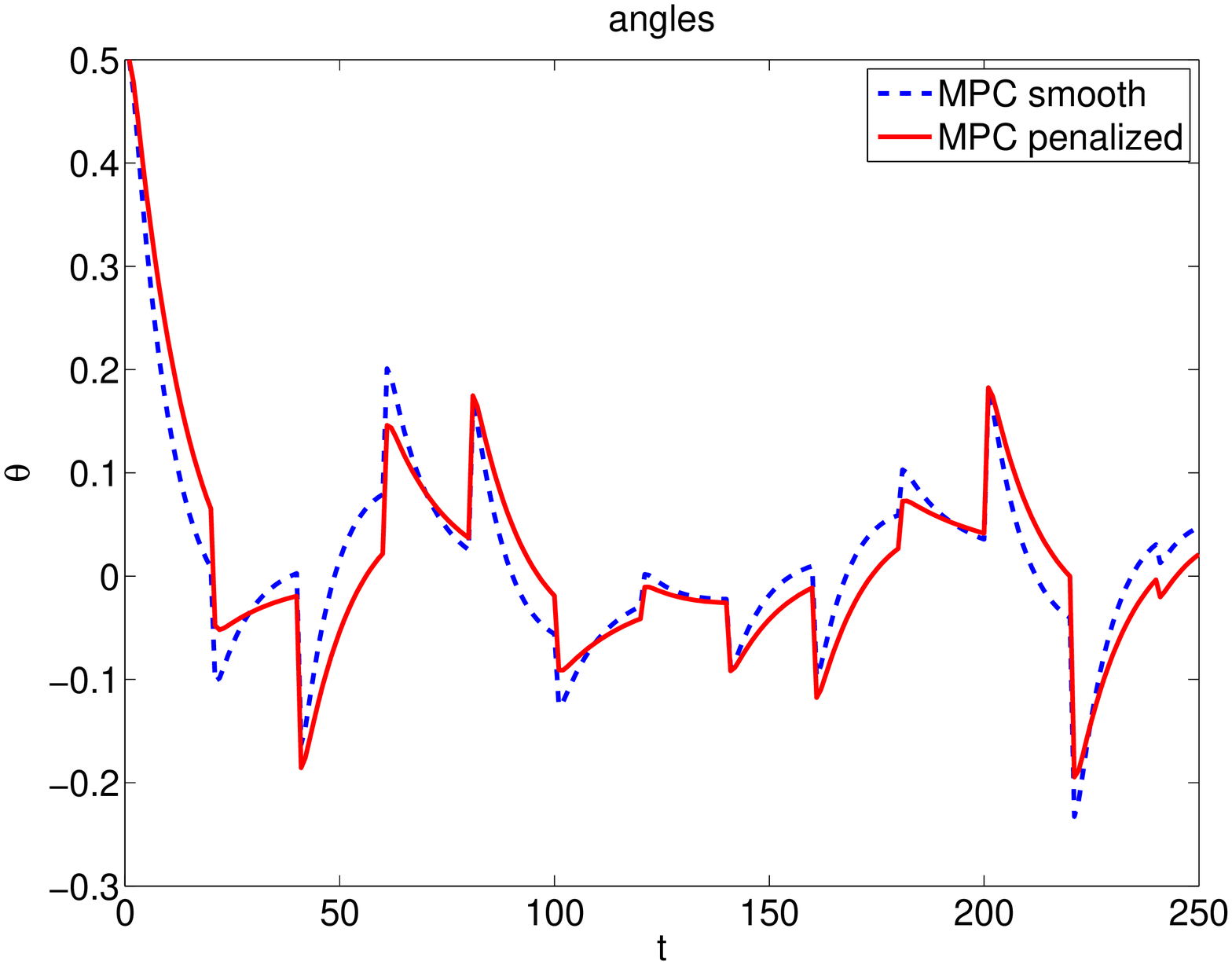}
\hspace{-0.5cm}
\includegraphics[width=0.5\textwidth,,height=5cm]{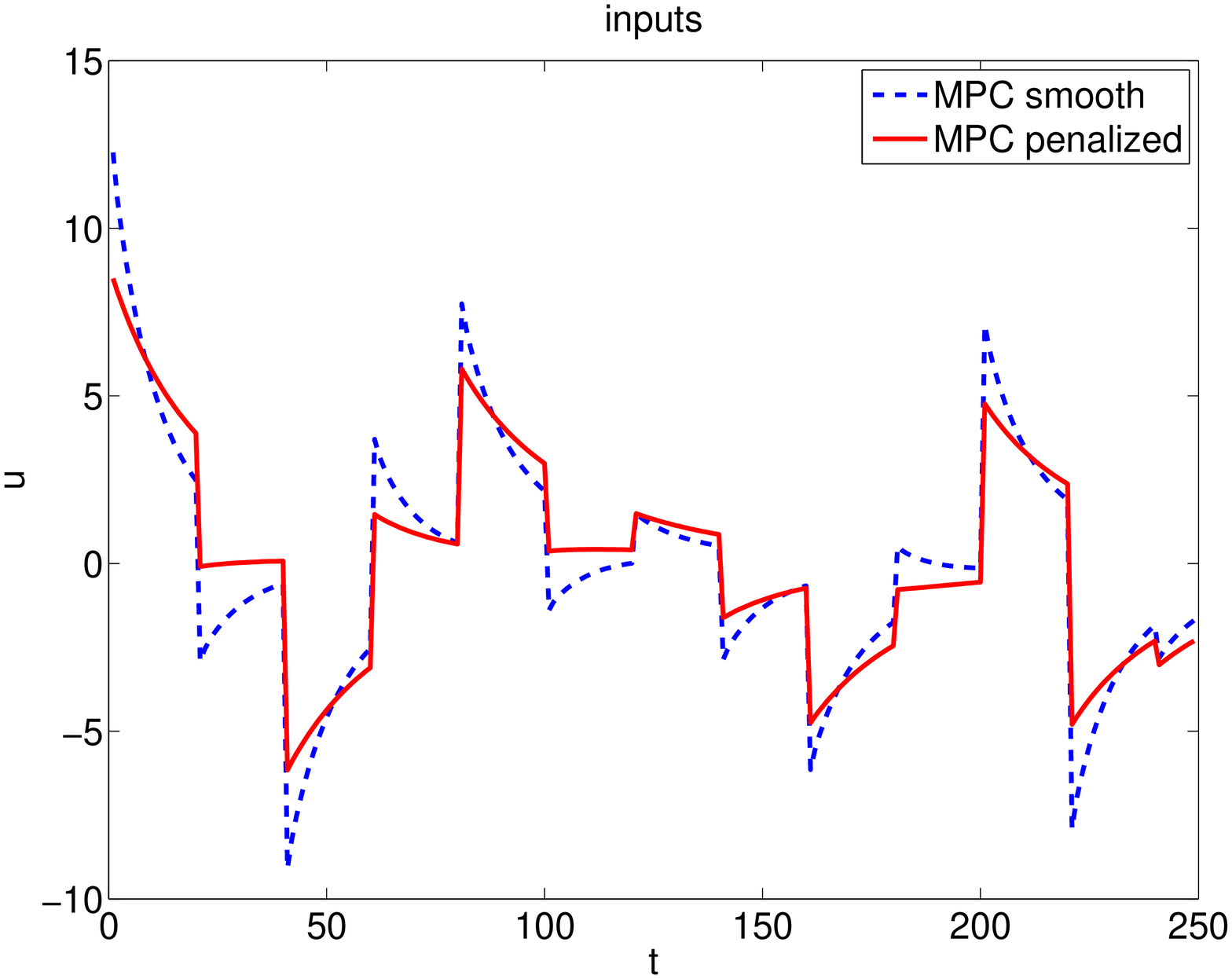}
\vspace*{-0.4cm}
\caption{The MPC trajectories of the state angle (left) and input
(right) for $N = 10$ obtained using algorithm {\bf IDGM} from DuQuad in the last
iterate with accuracy $\epsilon=10^{-2}$.}
\label{in:fig:states_inputs1}
\end{figure}

\noindent We consider a simplified model for the \an{self-balancing} Lego
mindstorm NXT extracted from \cite{Yam:08}. The model is linear time
invariant and stabilizable. The continuous linear model has the states $\mathbf{x} \in
\rset^4$ and inputs $\mathbf{u} \in \rset^{}$. The states for this system are the
horizontal position and speed ($h,\dot{h}$), and the angle to the
vertical and the angular velocity of the robot's body ($\theta,
\dot{\theta}$). The input for the system represents the pulse-width
modulaed voltage applied to both wheel motors in percentages. We
discretize the dynamical system via the zero-order hold method for a sample
time of $T=8 \text{ms}$ to obtain  the system matrices:
\begin{align*}
A&=\begin{bmatrix} 1 & 0.0054 & -2\cdot10^{-4} & 10^{-4} \\  0 & 0.4717 & -0.0465 & 0.0211
\\ 0 & 0.03 & 1.0049 & 0.0068 \\ 0 & 6.0742 & 1.0721 & 0.7633
\end{bmatrix}, \quad
B=\begin{bmatrix} 0.0002 \\ 0.0448 \\ -0.0025 \\ -0.5147
\end{bmatrix}.
\end{align*}

\noindent For this linear dynamical  system we consider the duty-cycle
percentage constraints for the inputs, i.e. $-12 \leq u(t) \leq 12$,  and additional
constraints for the position, i.e. $-0.5\leq h \leq 0.5$, and for the body
angle in degrees, i.e. $-15 \leq \theta \leq 15$. For the quadratic stage  cost
we consider matrices: $\mathbf{Q} = \text{diag} ([1 \; 1 \; 6\cdot10^2 \;1
])$ and $\mathbf{R }=  2$.

\noindent We consider two condensed  MPC formulations: \textit{MPC smooth}
and \textit{MPC penalized}, where we \an{impose} additionally a penalty
term  $\beta (u(t) - u(t-1))^2$, with $\beta=0.1$,
in order to get \an{a} smoother controller. Note
that in both formulations we obtain  QPs \cite{RawMay:09}.
Initial state is $\mathbf{x}=[0 \; 0\; 0.5 \; -0.35]^T$ and we add gentle
disturbances to the system at each $20$ simulation steps. In Fig.
\ref{in:fig:states_inputs1} we plot the MPC trajectories of the
state angle  and input  for a prediction horizon $N = 10$ obtained using
algorithm {\bf IDGM} in the last iterate with accuracy $\epsilon=10^{-2}$.
Similar state and input trajectories are obtained using the other versions of the scheme {\bf IDFOM} from DuQuad.
We observe a smoother behavior for MPC with \an{the} penalty term.


\end{document}